\documentclass[a4paper,11pt]{article}
\usepackage{amssymb,amsmath,amsthm,amsfonts}
\usepackage{bookmark}
\usepackage{mathrsfs}
\usepackage{multirow}
\usepackage{graphicx}
\usepackage{authblk}
\usepackage{indentfirst}
\usepackage{multicol}
\usepackage{tabu}
\usepackage{url}
\usepackage{fancyhdr}
\usepackage[numbers]{natbib}
\usepackage[all]{xy}
\usepackage{mathtools}
\usepackage[english]{babel}
\usepackage{colortbl}
\usepackage{caption}
\usepackage{hyperref}\usepackage{subcaption}
\usepackage{tikz}\usepackage{float}

\newtheorem{theorem}{Theorem}[section]
\newtheorem{proposition}[theorem]{Proposition}
\newtheorem{lemma}[theorem]{Lemma}

\newtheorem{example}{Example}[section]
\newtheorem{remark}{Remark}[section]
\newcommand{\im}{{\mathrm{im}\hspace{0.1em}}}

\newcommand{\sgn}{{\mathrm{sgn}\hspace{0.1em}}}
\newcommand{\qdim}{{\mathrm{\widehat{dim}}\hspace{0.1em}}}
\usepackage{xr}
\makeatletter
    \newcommand*{\addFileDependency}[1]{
    \typeout{(#1)}
    \@addtofilelist{#1}
    \IfFileExists{#1}{}{\typeout{No file #1.}}
    }
\makeatother

\newcommand{\leftcross}[1]{
  \begin{tikzpicture}[knot/.style={black}]
    \begin{scope}[xshift=5cm]
      \draw[line width=1pt] (-#1,#1)-- (#1,-#1);
      \draw[white,double=black,double distance=1pt,thick](-#1,-#1)-- (#1,#1);
    \end{scope}
  \end{tikzpicture}
}

\newcommand{\rightcross}[1]{
  \begin{tikzpicture}[knot/.style={black}]
    \begin{scope}[xshift=5cm]
      \draw[line width=1pt] (-#1,-#1)-- (#1,#1);
      \draw[white,double=black,double distance=1pt,thick](-#1,#1)-- (#1,-#1);
    \end{scope}
  \end{tikzpicture}
}

\newcommand{\udarc}[1]{
  \begin{tikzpicture}
    \draw[line width=1pt] (0,0) arc (140:40:#1cm);
    \draw[line width=1pt] (0,#1*1.2) arc (220:320:#1cm);
  \end{tikzpicture}
}

\newcommand{\lrarc}[1]{
  \begin{tikzpicture}
    \draw[line width=1pt] (#1*1.2,0) arc (130:230:#1cm);
    \draw[line width=1pt] (0,0) arc (50:-50:#1cm);
  \end{tikzpicture}
}

\title{Evolutionary Khovanov homology}

\author[1]{Li Shen}
\author[2,1]{Jian Liu}
\author[1,3,4]{Guo-Wei Wei \thanks{Corresponding author: weig@msu.edu}}
\affil[1]{Department of Mathematics, Michigan State University, MI 48824, USA}
\affil[2]{Mathematical Science Research Center, Chongqing University of Technology, Chongqing 400054, China}
\affil[3]{Department of Electrical and Computer Engineering, Michigan State University, MI 48824, USA}
\affil[4]{Department of Biochemistry and Molecular Biology, Michigan State University, MI 48824, USA}

\makeatletter
    \renewcommand*{\@fnsymbol}[1]{\ensuremath{\ifcase#1\or \dagger\or *\or *\or
   \mathsection\or \else\@ctrerr\fi}}
\makeatother
\date{}

\begin{document}
    \maketitle

    \paragraph{Abstract}
    Knot theory is a study of the embedding of closed circles into three-dimensional Euclidean space, motivated the ubiquity of knots in daily life and human civilization. However, the current knot theory focuses on the topology rather than metric. As such, the application of knot theory remains primitive and qualitative.
Motivated by the need of quantitative knot data analysis (KDA), this work implements the metric into knot theory, the evolutionary Khovanov homology (EKH), to facilitate a multiscale KDA of real-world data. It is demonstrated that EKH exhibits non-trivial knot invariants at appropriate scales even if the global topological structure of a knot is simple.
The proposed EKH has a great potential for KDA and knot learning.



    \newpage
    \tableofcontents
    \newpage

\section{Introduction}\label{section:introduction}

Knots are interlaced structures formed by tying a piece of rope, string, or other flexible material and are omnipresent, from practical uses in sailing, climbing, and fishing to decorative purposes in crafts and art. Knot structures are pervasive, being studied in various fields, including   physics \cite{atiyah1990geometry},  chemistry \cite{lukin2005knotting}, biology \cite{murasugi1996knot}, bioengineering \cite{endy2005foundations}, and therapeutics \cite{pommier2010dna}.
Knot theory is a branch of mathematics that studies knot invariants and has a long history \cite{adams1994knot,burde2002knots}. The most well-known knot invariant is the Jones polynomial \cite{jones1985polynomial}. Around the year 2000, Khovanov introduced the Khovanov homology, which provides an algebraic topology theory of knots. Notably, the graded Euler characteristic of Khovanov homology coincides with the Jones polynomial \cite{bar2002khovanov,khovanov2000categorification}. Khovanov homology is one of the most significant achievements in knot theory and has stimulated much development in the past two decades. 
However, the current knot theory is entirely topological and lacks  metric space analysis.
As such, the applications of knot theory are mostly qualitative or limited to global information about knot types \cite{goundaroulis2017topological,lim2015molecular}.

Recently, multiscale  Gauss link integral (MGLI) theory has been proposed for quantitative  Knot data analysis (KDA) and knot learning (KL) \cite{shen2023knot}. MGLI outperforms other methods, including topological data analysis (TDA),  on biomolecular datasets, demonstrating the great promise of KDA and KL \cite{shen2023knot}. However, MGLI does not offer a topological analysis at various small scales, although it preserves the global knot topology at a large scale.

Real-world knots often exhibit highly complex geometric structures, while their topological structures may be relatively simple. For example, the helical structure of DNA is extremely intricate, but some helices can be simplified into simpler knot structures through transformations such as Reidemeister moves. This indicates that characterizing biological knots solely by their topological structure is often insufficient. On the other hand, considering the geometric details of knots in biological systems can be overly detailed and may obscure the essential properties of these structures. Therefore, a balance between geometry and topology must be struck to capture both the essential geometric details and  the fundamental topological structures of knots.

Persistent homology facilities TDA and has been  used to capture the topological features of point clouds by continuously analyzing its topological structures at different scales \cite{carlsson2009topology,edelsbrunner2008persistent}. This technique has found widespread applications in various fields such as biology, materials science, and geographic information systems \cite{cang2017topologynet}. The basic idea involves representing a point cloud as a distribution of points in space and observing the evolution of its topological structure by systematically changing the geometric scale. This approach motivates a new  knot theory, namely, multiscale (persistent) knot theory. This approach allows us to embed knot geometric shapes into topological invariants, enabling a more nuanced exploration of their geometric properties and topological characteristics.

In this work, we introduce the concept of Evolutionary Khovanov Homology (EKH) to investigate the geometric and topological properties of knots and links in Euclidean space. Specifically, we explore a filtration of links by considering smoothing transformations of crossings in knots. These links provide a sequence illustrating the evolution of Khovanov homology, thereby revealing finer topological structures and geometric shapes of the links. It is worth noting that while the Khovanov homology of unknotted links is trivial, their evolutionary Khovanov homology may not necessarily be so (see  Example \ref{example:unknot}). Consequently, evolutionary Khovanov homology offers a promising framework for characterizing knots or links in real data, which may possess complex structures yet exhibit simplicity under knot equivalence relations. In addition, we employ evolutionary Khovanov homology to investigate the three-dimensional (3D) structure of a SARS-CoV-2 frameshifting pseudoknot. Even though the corresponding knot diagram of the abstract knot of the  pseudoknot is unknotted, the evolutionary Khovanov homology extracts significant topological and geometric information. EKH provides a new KDA paradigm for real world data and can be used for knot deep learning (KDL).

The structure of the paper is as follows. The next section  knot theory, encompassing important concepts such as knot invariants, Gauss code, the Jones polynomial, and Khovanov homology, is reviewed . In Section \ref{section:evolutionary}, we introduce the concept of evolutionary Khovanov homology and provide computations and representations of examples. Section \ref{section:application} demonstrates the application of evolutionary Khovanov homology in biology. Finally, this paper end s with a conclusion.

\section{Knot theory}

To establish notations, we review some fundamental concepts of knot theory in this section, including Reidemeister moves, knot invariants, Gauss code, Kauffman bracket, Jones polynomial, and Khovanov homology. We aim to present these topics in a self-contained manner. For readers interested in a more detailed study of knot theory, we recommend the references \cite{adams1994knot,burde2002knots}.

\subsection{Knot invariant}

A \emph{knot} is an embedding of the circle $S^{1}$ into three-dimensional Euclidean space $\mathbb{R}^{3}$ or into the 3D sphere $S^{3}$. Sometimes, the knot is required to be piecewise smooth and to have a non-vanishing derivative on each closed interval.

Two embeddings $f,g:N\to M$ of manifolds are called \emph{ambient isotopy} if there is a continuous map $F:M\times [0,1]\to M$ such that if $F_{0}$ is the identity map, each $F_{t}:M\to M$ is a homeomorphism, and $F_{1}\circ f=g$.

Two knots are \emph{equivalent} if there is an ambient isotopy between them. It is one of the pivotal challenges in knot theory to study the equivalence classes of knots. This equivalence allows us to systematically study the properties and characteristics of knots without considering their specific shapes or spatial positions. Based on this, researchers have developed various knot invariants and established the topology of knots.

A knot in $\mathbb{R}^{3}$ (resp. $S^{3}$) can be projected into the Euclidean plane $\mathbb{R}^{2}$ (resp. $S^{2}$). From now on, unless specifically stated otherwise, we will focus on knots in $\mathbb{R}^{3}$. For knots in $S^{3}$, we can provide analogous descriptions.

A projection $p:K\to \mathbb{R}^{2}$ of a knot $K$ is \emph{regular} if it is injective everywhere, except at a finite number of crossing points. These crossing points are the projections of double points of the knot, and should occur only where lines intersect. Moreover, the crossing points contain the information of overcrossings and undercrossings. Such a projection is commonly referred to as a \emph{knot diagram}.

It is worth noting that a knot can have different regular projections. Consequently, for a given knot, we can obtain different knot diagrams. Indeed, the knot diagram is independent of the choice of projection up to equivalence. Before proceeding, let us recall the Reidemeister moves.

The \emph{Reidemeister moves} are the following three operations on a small region of the diagram:
\begin{itemize}
  \item[(R1)] Twist and untwist in either direction;
  \item[(R2)] Move one loop completely over Or under another;
  \item[(R3)] Move a string completely over or under a crossing.
\end{itemize}
\begin{figure}[htb!]
  \centering
  \includegraphics[width=0.8\textwidth]{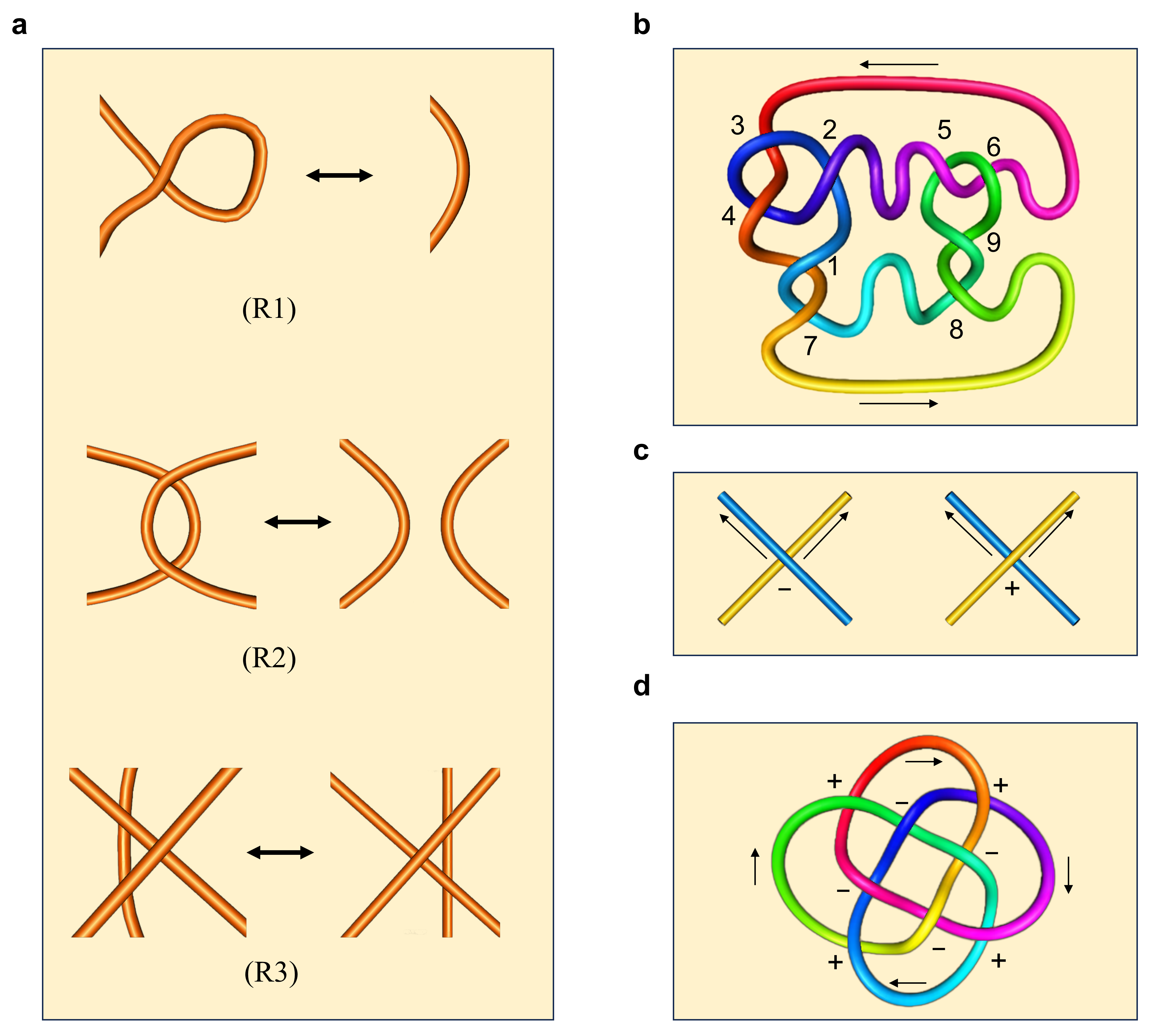}
  \caption{\textbf{a} The three types of Reidemeister moves. \textbf{b} The marked diagram of a knot can be used to obtain the Gauss code. \textbf{c} The left is the left-handed crossing, and the right is the right-handed crossing. \textbf{d} The knot with crossings marked by $+$ or $-$. The corresponding writhe number is $w(L)=4-4=0$.}
	\label{figure:knot_combination}
\end{figure}
K. Reidemeister et al. have shown that two knot diagrams belonging to the same knot can be transformed into each other by a sequence of the three Reidemeister moves up to ambient isotopy \cite{alexander1926types,reidemeister1927elementare}. Moreover, two knots are equivalent if and only if all their projections are equivalent \cite{burde2002knots}. This suggests that the equivalence relation of knots can be established using Reidemeister moves, which are more user-friendly compared to ambient isotopy. They also facilitate proving whether a quantity is a knot invariant.

A \emph{knot invariant} is a quantity defined on knots that remains unchanged under knot equivalence. The most common knot invariants include tricoloring \cite{kauffman1987state}, crossing number \cite{adams1994knot}, bridge number \cite{schubert1954numerische}, and the Jones polynomial \cite{jones1985polynomial}. However, these knot invariants cannot determine the equivalent class of knots; indeed, it is even difficult to determine if a knot is the trivial knot. This underscores the inadequacy of current knot invariants, prompting ongoing efforts to seek new ones. Among these knot invariants, the Jones polynomial stands out as one of the most successful. It encapsulates critical information regarding knot topology and structure, including symmetry, crossing distribution, and complexity. Furthermore, its profound links to fields such as topological quantum field theory and quantum braid theory in physics underscore its importance in understanding topological phase transitions and quantum states.

\subsection{Gauss code}

The Gauss code represents a knot diagram using a sequence of integer numbers \cite{gibson2011homotopy}. This digital representation facilitates recording and understanding of the knot diagram. Moreover, we can reconstruct the original knot diagram from its Gauss code. This implies that Gauss code holds significant importance in classifying knots and computing knot invariants.

Given a knot diagram $K$, one can obtain a Gauss code $G(K)$ as follows:
\begin{enumerate}
  \item Choose a crossing as the starting point and select a direction to begin from the starting point;
  \item Assign the starting crossing a value of 1, and then assign values of 2, 3, and so on to each subsequent unlabeled crossing along the chosen direction;
  \item For each crossing, we assign a sign. If the crossing is an overcrossing, the sign is positive; otherwise, it is negative.
\end{enumerate}
The integer sequence written down following the aforementioned procedure is what we refer to as the Gauss code. For example, see Figure \ref{figure:knot_combination}\textbf{b}. Starting from 1 and proceeding to 2, we obtain a sequence of numbers, denoted as $1,2,3,4,2,5,6,3,4,1,7,8,9,6,5,9,8,7$. By assigning a sign to each number based on the type of crossing, we get a new sequence of numbers:
\begin{equation*}
  +1,-2,+3,-4,+2,+5,-6,-3,+4,-1,+7,+8,-9,+6,-5,+9,-8,-7.
\end{equation*}
This sequence is the Gauss code for the knot in Figure \ref{figure:knot_combination}\textbf{b}.

For a Gauss code $C$, we can reconstruct a knot diagram $D(C)$. So, the natural question arises: for a knot diagram $K$, is the knot diagram $D(G(K))$ equivalent to $K$? In general, this is not entirely correct. To address this issue, people have introduced extended Gauss code. The construction of the extended Gauss code is similar to the Gauss code, with one key difference in how the signs of the integers are assigned. When the crossing is right-handed, the integer is assigned a positive value, and when it is left-handed, the integer is assigned a negative value. For Figure \ref{figure:knot_combination}\textbf{b}, by considering the right-handed or left-handed nature of each crossing, we obtain the extended Gauss code:
\begin{equation*}
\begin{split}
    & +1L,-2R,+3R,-4R,+2R,+5L,-6L,-3R,+4R,-1L,\\
    & +7L,+8R,-9R,+6L,-5L,+9R,-8R,-7L.
\end{split}
\end{equation*}

In theory, Gauss code helps us examine and understand information about knots, which allows us to study their properties. In computation, Gauss code can be utilized to calculate various knot invariants, such as the Jones polynomial, Alexander polynomial, and others. Furthermore, from an algorithmic perspective, digitizing and processing knot data through Gauss code are invaluable for computer-assisted knot research and computation.

\subsection{Kauffman bracket and Jones polynomial}

In the previous section, we concluded that to study the invariants of knots, it is sufficient to explore the invariance of knot diagrams under Reidemeister moves. From now on, our attention will be directed towards knot diagrams as we revisit the Kauffman bracket and Jones polynomial associated with them.

For a crossing, there is a 0-smoothing \raisebox{-0.1cm}{\udarc{0.3}} and a 1-smoothing \raisebox{-0.15cm}{\lrarc{0.3}}. The process of smoothing can be understood as untangling a crossing, as illustrated below.
\begin{eqnarray*}\label{equation:smoothing}
  \raisebox{-0.13cm}{\leftcross{0.2}} &\Longrightarrow& \raisebox{-0.1cm}{\udarc{0.3}}\quad +\quad \raisebox{-0.15cm}{\lrarc{0.3}} \\
  \raisebox{-0.13cm}{\rightcross{0.2}} &\Longrightarrow&\raisebox{-0.1cm}{\udarc{0.3}}\quad +\quad \raisebox{-0.15cm}{\lrarc{0.3}}
\end{eqnarray*}
    \indent A \emph{link} is a collection of knots that do not intersect but may be linked (or knotted) together. In particular, a knot is a link with only one component. If not explicitly stated, the links discussed in this paper are assumed to be orientable.

Given a knot $K$ and a crossing $x$ of $K$, we can create links by replacing the crossing $x$ with the 0-smoothing and the 1-smoothing, respectively. Let $\mathbf{Knot}$ denote the set of knots, and let $\mathbf{Link}$ denote the set of links. Given a link $L$, let $\mathcal{X}(L)$ denote the set of crossings of $L$. For each $x\in \mathcal{X}(L)$, the smoothing operators at $x$ lead to the 0-smoothing and the 1-smoothing maps $\rho_{0}, \rho_{1}: \mathbf{Link} \to \mathbf{Link}$ as $L \mapsto \rho_{0}(L,x)$ and $L \mapsto \rho_{1}(L,x)$, respectively. In the following construction of the Kauffman bracket, for unoriented knot, the smoothing is always performed on the undercrossing $\raisebox{-0.13cm}{\rightcross{0.2}}$.

The \emph{Kauffman bracket} is a bracket function $\langle -\rangle: \mathbf{Link}\to \mathbb{Z}[a,a^{-1}]$ satisfying:
\begin{enumerate}
  \item[$(a)$] $\langle \bigcirc\rangle=1$;
  \item[$(b)$] $\langle \bigcirc\cup L\rangle=(-a^{2}-a^{-2})\langle L\rangle$;
  \item[$(c)$] $\langle L\rangle =a\langle \rho_{0}(L,x)\rangle+a^{-1} \langle \rho_{1}(L,x)\rangle$ for any $x\in \mathcal{X}(L)$.
\end{enumerate}
Here, $\bigcirc$ denotes the trivial knot.

The Kauffman bracket does always exist, and it is uniquely determined in $\mathbb{Z}[a,a^{-1}]$. Now, let $n = |\mathcal{X}(L)|$ be the number of crossings of $L$. For each crossing, we have the options of performing 0-smoothing and 1-smoothing. Thus, we can obtain a total of $2^n$ different smoothing links. Each of these smoothing links is referred to as a \emph{state} of the link $L$. All the states together form a state cube.  Another description of the Kauffman bracket is given in terms of the state cube of a link \cite{kauffman1987state}. For a state $s$ of $L$, let $\alpha(s)$ and $\beta(s)$ denote the number of 0-smoothings and 1-smoothings of crossings in state $s$, respectively. The Kauffman bracket is
\begin{equation}
  \langle L\rangle = \sum\limits_{s}(-1)^{\alpha(s)-\beta(s)}(-a^{2}-a^{-2})^{\gamma(s)-1}.
\end{equation}
Here, $s$ runs through all the states of $L$, and $\gamma(s)$ is the number of circles of $L$ in the state $s$.

It is worth noting that the Kauffman bracket is invariant under the Reidemeister moves (R2) and (R3). However, the Kauffman bracket is not a knot invariant, as it is not invariant under (R1). To define a knot invariant, we first introduce the concept of the writhe number. Consider an oriented diagram of a link $L$. Let's define $w(L)$ as follows: with each crossing of $L$, we associate $+1$ if it is a right-handed crossing, and $-1$ if it is a left-handed crossing. For an example, see Figure \ref{figure:knot_combination}\textbf{c} and \textbf{d}. By summing these numbers at all crossings, we obtain the writhe number $w(L)$.

The \emph{Kauffman polynomial} (or normalized Kauman bracket) of a link $L$ is the polynomial defined as follows
\begin{equation}
  X_{L}(a)=(-a)^{-3w(L)}\langle L\rangle.
\end{equation}
The Kauffman polynomial is a knot invariant \cite{manturov2018knot}. By substituting $a$ in $X_{L}(t)$ with $t^{-\frac{1}{4}}$, we obtain the \emph{Jones polynomial}
\begin{equation}
  V_{L}(t)=X_{L}(t^{-\frac{1}{4}}).
\end{equation}
The Jones polynomial is a famous knot invariant introduced by V. Jones \cite{jones1985polynomial}.
\begin{remark}
With the previous notations, if we set $q = -a^{-2}$, then the Kauffman bracket can be described by the conditions
\begin{enumerate}
  \item[$(a')$] $\langle \bigcirc\rangle=q+q^{-1}$;
  \item[$(b')$] $\langle \bigcirc\cup L\rangle=(q+q^{-1})\langle L\rangle$;
  \item[$(c')$] $\langle L\rangle =\langle \rho_{0}(L,x)\rangle-q\langle \rho_{1}(L,x)\rangle$ for any $x\in \mathcal{X}(L)$.
\end{enumerate}
Let $n_{+}$ be the number of right-handed crossings in $\mathcal{X}(L)$, and let $n_{-}$ be the number of left-handed crossings in $\mathcal{X}(L)$. The unnormalized Jones polynomial is defined by
\begin{equation}
  \hat{J}(L)=(-1)^{n_{-}}q^{n_{+}-2n_{-}}\langle L\rangle.
\end{equation}
Then the Jones polynomial of $L$ is defined as $J(L)= \hat{J}(L)/(q+q^{-1})$. This definition is more convenient for categorifying the Jones polynomial, as specifically detailed in the literature \cite{khovanov2000categorification}.
\end{remark}

\begin{figure}[htb!]
	\centering
	\includegraphics[width=0.9\textwidth]{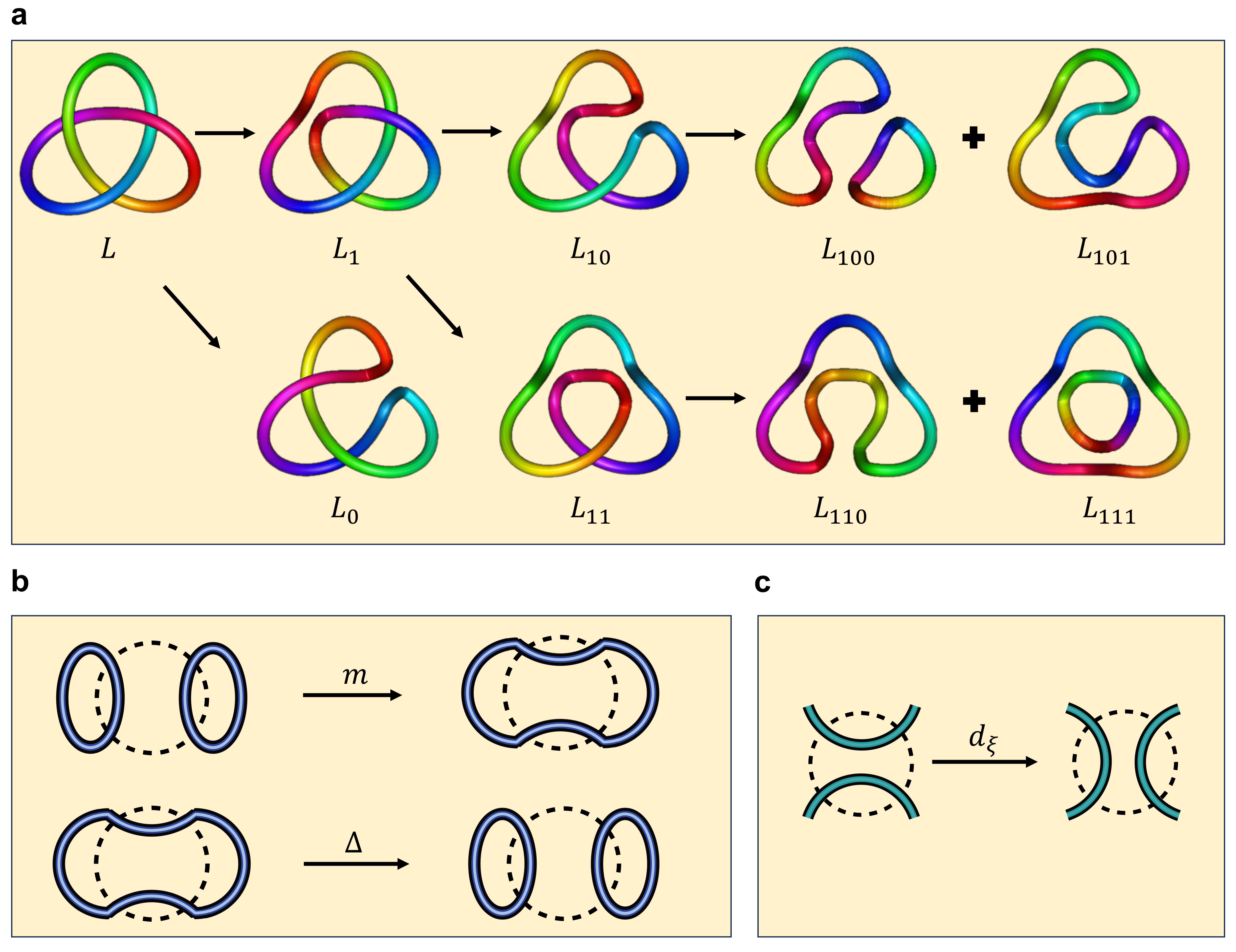}\\
	\caption{{\bf a} The links by conducting $0$-smoothings and $1$-smoothings of the undercrossings of a left-handed trefoil.
	{\bf b} Two circles merging into one, or one circle splitting into two.
	{\bf c} Illustration of differential.
	}\label{figure:knot_differential}
\end{figure}

\begin{example}\label{example:jones}
Let $L$ be a left-handed trefoil. Consider the smoothing of $L$ shown in Figure \ref{figure:knot_differential}\textbf{a}. For example, the link $L_{100}$ represents the original link after performing one $1$-smoothing, followed by two $0$-smoothings. Note that
\begin{equation*}
  \begin{split}
    \langle L_{100}\rangle &=  \langle \bigcirc\cup  \bigcirc\rangle= (q+q^{-1})^{2}, \\
    \langle L_{101}\rangle &=\langle\bigcirc\rangle= (q+q^{-1}),\\
    \langle L_{110}\rangle &=\langle\bigcirc\rangle=(q+q^{-1}),\\
    \langle L_{111}\rangle &=\langle\bigcirc \cup \bigcirc\rangle= (q+q^{-1})^{2}.
  \end{split}
\end{equation*}
It follows that
\begin{equation*}
  \begin{split}
  \langle L_{10}\rangle= \langle L_{100}\rangle-q\langle L_{101}\rangle=q^{-1}(q+q^{-1}),\\
  \langle L_{11}\rangle= \langle L_{110}\rangle-q\langle L_{111}\rangle =-q^{2}(q+q^{-1}).
  \end{split}
\end{equation*}
Thus we have $\langle L_{1}\rangle=\langle L_{10}\rangle-q\langle L_{11}\rangle=(q^{-1}+q^{3})(q+q^{-1})$. By a similar calculation, we can obtain $\langle L_{0}\rangle=q^{-2}(q+q^{-1})$.
Hence, we obtain
\begin{equation*}
  \langle L\rangle=\langle L_{0}\rangle-q\langle L_{1}\rangle=(q^{-2}-1-q^{4})(q+q^{-1}).
\end{equation*}
Thus the unnormalized Jones polynomial of $L$ is
\begin{equation*}
  \hat{J}(L)=(-1)^{3}q^{-6}\langle L\rangle=q^{-1}+q^{-3}+q^{-5}-q^{-9}.
\end{equation*}
And the Jones polynomial of $L$ is $q^{-2}+q^{-6}-q^{-8}$.
\end{example}

\subsection{Khovanov homology}

Khovanov homology was introduced by Mikhail Khovanov around the year 2000. It is regarded as a categorification of the Jones polynomial, providing a topological interpretation of the Jones polynomial \cite{bar2002khovanov,manturov2018knot}. Specifically, the graded Euler characteristic of Khovanov homology corresponds to the Jones polynomial. Compared to the Jones polynomial, Khovanov homology contains more information. Notably, Khovanov homology can detect the unknot \cite{kronheimer2011khovanov}.

\textbf{Graded dimension} Let $V=\sum\limits_{k\in \mathbb{Z}}V_{k}$ be a graded vector space. The graded dimension of $V$ is the power series
   \begin{equation*}
    \qdim  V=\sum\limits_{k\in \mathbb{Z}}q^{k}\dim V_{k}.
   \end{equation*}
For example, if $V$ is generated by three elements $v_{-1},v_{0},v_{1}$ with the grading $-1$, $0$, $1$, respectively, then the graded dimension of $V$ is $q^{-1}+1+q$.

\textbf{Degree shift} The degree shift on a graded vector space $V=\sum\limits_{k\in \mathbb{Z}}V_{k}$ is an operation $\cdot \{l\}$ such that $W\{l\}_{k}=W_{k-l}$. By definition, one has that
  \begin{equation*}
   \qdim  V\{l\}=q^{l}\qdim  V.
  \end{equation*}
\indent \textbf{Height shift} Let $C$ denote the cochain complex $\cdots \to C^{n}\stackrel{d^{n}}{\to}C^{n+1}\to \cdots$. The height shift of $C^{\ast}$ is the operation $\cdot [m]$ such that
  $C[m]$ is a cochain complex with $C[m]^{n}=C^{n-m}$ and $d[m]^{n}=d^{n-m}:C^{n-m} \to C^{n-m+1}$.

Recall that for a link, we have a state cube $\{0,1\}^{\mathcal{X}(L)}$. Each state $s$ in $\{0,1\}^{\mathcal{X}(L)}$ can be represented as $(s_{1},s_{2},\dots,s_{n})$, where $n=|\mathcal{X}(L)|$. Now, let $\mathbb{K}$ be the ground field, and let $V$ be a graded vector space with two generators $v_{-},v_{+}$. Then $\qdim V=q^{-1}+q$. For each state $s\in \{0,1\}^{\mathcal{X}(L)}$, we have a space $V_{s}(L)=V^{\otimes c(s)}\{\ell(s)\}$, where $c(s)$ is the number of circles in the smoothing of $L$ at state $s$, and $\ell (s)=\sum\limits_{i=1}^{n} s_{i}$ is the number of ones in the representation of $s$. The $k$-th chain group of $L$ is defined
\begin{equation}
 [[L]]^{k}:=\bigoplus\limits_{s:\ell(s)=k}  V_{c(s)}(L).
\end{equation}
Then $[[L]]$ is a graded vector space. Furthermore, we can obtain a cochain complex $[[L]]\{n_{+}-2n_{-}\}$. The Khovanov chain group of $L$ is defined by
\begin{equation}
   \mathcal{C}(L):=[[L]][-n_{-}]\{n_{+}-2n_{-}\}.
\end{equation}
More precisely, we have
\begin{equation}
  \mathcal{C}^{k}(L)=\bigoplus\limits_{\ell(s)=k+n_{-}}  V^{\otimes c(s)}\{\ell(s)+n_{+}-2n_{-}\}.
\end{equation}
Note that $\mathcal{C}^{k}(L)$ itself is a graded vector space. Thus there is a natural graded structure on $\mathcal{C}^{k}(L)$. To obtain a cochain complex, we will endow $\mathcal{C}(L)$ with a differential as follows. Consider the state cube $\{0,1\}^{\mathcal{X}(L)}$ with $n\cdot 2^{n-1}$ edges. Each of the edges is of the form
\begin{equation*}
  (s_{1},s_{2},\dots,s_{i-1},0,s_{i+1},\dots,s_{n})\to (s_{1},s_{2},\dots,s_{i-1},1,s_{i+1},\dots,s_{n}).
\end{equation*}
We denote the edge by $\xi=(\xi_{1},\xi_{2},\dots,\xi_{i-1},\star,\xi_{i+1},\dots,\xi_{n})$. Let $\sgn(\xi)=(-1)^{\xi_{1}+\cdots+\xi_{i-1}}$, and let $|\xi|=\sum\limits_{t\neq i}\xi_{t}$. The differential $d^{k}:\mathcal{C}^{k}(L)\to \mathcal{C}^{k+1}(L)$ is defined by $d=\sum\limits_{|\xi|=k}\sgn(\xi)\cdot d_{\xi}$. Now we will review the construction of $d_{\xi}$. Note that an edge of the state cube connects two adjacent states. The two states differ by just one crossing's smoothing, which implies that the diagrams corresponding to these two states differ by just one circle. Geometrically, this is manifested as two circles merging into one, or one circle splitting into two, see Figure \ref{figure:knot_differential}\textbf{b} and \textbf{c}.

Algebraically, the above process can be understood as $V\otimes V\to V$ or $V\to V\otimes V$, because the word length of the term $V^{\otimes c(s)}\{\ell(s)+n_{+}-2n_{-}\}$ is equal to the number of circles. The map $d_{\xi}:\mathcal{C}^{k}(L)\to \mathcal{C}^{k+1}(L)$ is defined as
\begin{equation}\label{equation:multiplication}
  m:V\otimes V\to V,\quad m:\left\{
                             \begin{array}{ll}
                               v_{+}\otimes v_{+}\mapsto v_{+},\quad v_{-}\otimes v_{+}\mapsto v_{-}, \\
                               v_{+}\otimes v_{-}\mapsto v_{-},\quad v_{-}\otimes v_{-}\mapsto 0
                             \end{array}
                           \right.
\end{equation}
on the components involved in merging,
\begin{equation}\label{equation:comultiplication}
  \Delta:V\to V\otimes V,\quad \Delta:\left\{
                             \begin{array}{ll}
                               v_{+} \mapsto v_{+}\otimes v_{-}+v_{-}\otimes v_{+}, \\
                               v_{-}\mapsto v_{-}\otimes v_{-}
                             \end{array}
                           \right.
\end{equation}
on the components involved in splitting, and identity at other components. It can be verified that the above construction indeed provides a differential structure on $\mathcal{C}(L)$. Therefore, $\mathcal{C}(L)$ is a cochain complex, called the \emph{Khovanov complex}. The \emph{Khovanov (co)homology} of $L$ is defined by
\begin{equation*}
  H^{k}(L):=H^{k}(\mathcal{C}(L)),\quad k\geq 1.
\end{equation*}
The Khovanov homology is a well-known knot invariant, which can decode the Jones polynomial. We call the rank of $H^{k}(L)$ the \emph{$k$-th Betti polynomial} of $L$, denoted by $\beta_{k}(q)$.

The \emph{graded Poincar\'{e} polynomial} of $\mathcal{C}(L)$ is defined by
\begin{equation}
  Kh(L)=\sum\limits_{k}\qdim H^{k}(L)\cdot t^{k}.
\end{equation}
By taking $t=-1$, we have the \emph{graded Euler characteristic} of $L$ given by
\begin{equation}
  \mathcal{X}_{q}(L)=\sum\limits_{k}(-1)^{k}\qdim H^{k}(L).
\end{equation}
It is worth noting that $\mathcal{X}_{q}(L)=\sum\limits_{k}(-1)^{k}\qdim \mathcal{C}^{k}(L)$.
A famous result assert that the graded Euler characteristic of $L$ equals to the unnormalized Jones polynomial of $L$.
\begin{theorem}
Let $L$ be a link. We have $\mathcal{X}_{q}(L)=\hat{J}(L)$.
\end{theorem}
The above result demonstrates that Khovanov homology provides a categorical interpretation of the Jones polynomial, thereby establishing the significant role of Khovanov homology in knot theory.
In this work, our focus lies in applying the features of Khovanov homology to analyze and study knots with spatial twists. Persistence is the core principle in analyzing the spatial geometric structure of knots. This prompts us to investigate evolutionary Khovanov homology in subsequent sections.

\begin{example}
Let $L$ be the left-handed trefoil. All the crossings are left-handed. Then we have the Khovanov cochain complex of $L$ given by
\begin{equation*}
  \xymatrix{
 0\ar@{->}[r]&\mathcal{C}^{-3}(L)\ar@{->}[r]^{d^{-3}}&\mathcal{C}^{-2}(L)\ar@{->}[r]^{d^{-2}}&\mathcal{C}^{-1}(L)\ar@{->}[r]^{d^{-1}}&\mathcal{C}^{0}(L)\ar@{->}[r]&0.
 }
\end{equation*}
Here, the space $\mathcal{C}^{k}(L)$ are obtained by the circles of states listed as follows.
\begin{equation*}
\begin{split}
    \mathcal{C}^{-3}(L)=&\overbrace{V\otimes V\otimes V}^{(0,0,0)}, \\
    \mathcal{C}^{-2}(L)=&\overbrace{V\otimes V}^{(1,0,0)}\oplus \overbrace{V\otimes V}^{(0,1,0)}\oplus \overbrace{V\otimes V}^{(0,0,1)},\\
    \mathcal{C}^{-1}(L)=&\overbrace{V}^{(1,1,0)}\oplus \overbrace{V}^{(1,0,1)}\oplus \overbrace{V}^{(0,1,1)},\\
    \mathcal{C}^{0}(L)=&\overbrace{V\otimes V}^{(1,1,1)}.
\end{split}
\end{equation*}
Recall that $V$ has two generators $v_{+}$ and $v_{-}$. Thus the space $\mathcal{C}^{-3}(L)$ has the basis
\begin{equation*}
  \begin{split}
    &v_{+}\otimes v_{+}\otimes v_{+}, v_{+}\otimes v_{+}\otimes v_{-}, v_{+}\otimes v_{-}\otimes v_{+}, v_{-}\otimes v_{+}\otimes v_{+},\\
    &v_{+}\otimes v_{-}\otimes v_{-}, v_{-}\otimes v_{+}\otimes v_{-}, v_{-}\otimes v_{-}\otimes v_{+}, v_{-}\otimes v_{-}\otimes v_{-},
\end{split}
\end{equation*}
the space $\mathcal{C}^{-2}(L)$ has the basis
\begin{equation*}
\begin{split}
  &(v_{+}\otimes v_{+},0,0),(v_{+}\otimes v_{-},0,0),(v_{-}\otimes v_{+},0,0),(v_{-}\otimes v_{-},0,0), \\
  &(0,v_{+}\otimes v_{+},0),(0,v_{+}\otimes v_{-},0),(0,v_{-}\otimes v_{+},0),(0,v_{-}\otimes v_{-},0), \\
  &(0,0,v_{+}\otimes v_{+}),(0,0,v_{+}\otimes v_{-}),(0,0,v_{-}\otimes v_{+}),(0,0,v_{-}\otimes v_{-}),
\end{split}
\end{equation*}
the space $\mathcal{C}^{-1}(L)$ is generated by
\begin{equation*}
  (v_{+},0,0),(v_{-},0,0),(0,v_{+},0),(0,v_{-},0),(0,0,v_{+}),(0,0,v_{-}),
\end{equation*}
and the space $\mathcal{C}^{0}(L)$ has the basis
\begin{equation*}
  v_{+}\otimes v_{+},v_{+}\otimes v_{-},v_{-}\otimes v_{+},v_{-}\otimes v_{-}.
\end{equation*}
Thus the left representation matrix of the differentials $d^{-3},d^{-2},d^{-1}$ with respect to the chosen basis are given by
\begin{equation*}
  B_{-3}=\left(
          \begin{array}{cccccccccccc}
            1 & 0 & 0 & 0 & 1 & 0 & 0 & 0 & 1 & 0 & 0 & 0 \\
            0 & 1 & 0 & 0 & 0 & 1 & 0 & 0 & 0 & 1 & 0 & 0 \\
            0 & 0 & 1 & 0 & 0 & 1 & 0 & 0 & 0 & 0 & 1 & 0 \\
            0 & 0 & 1 & 0 & 0 & 0 & 1 & 0 & 0 & 1 & 0 & 0 \\
            0 & 0 & 0 & 1 & 0 & 0 & 0 & 0 & 0 & 0 & 0 & 1 \\
            0 & 0 & 0 & 1 & 0 & 0 & 0 & 1 & 0 & 0 & 0 & 0 \\
            0 & 0 & 0 & 0 & 0 & 0 & 0 & 1 & 0 & 0 & 0 & 1 \\
            0 & 0 & 0 & 0 & 0 & 0 & 0 & 0 & 0 & 0 & 0 & 0 \\
          \end{array}
        \right),  B_{-2}=\left(
          \begin{array}{cccccc}
            -1 & 0 & -1 & 0 & 0 & 0 \\
            0 & -1 & 0 & -1 & 0 & 0 \\
            0 & -1 & 0 & -1 & 0 & 0 \\
            0 & 0 & 0 & 0 & 0 & 0 \\
            1 & 0 & 0 & 0 & -1 & 0 \\
            0 & 1 & 0 & 0 & 0 & -1 \\
            0 & 1 & 0 & 0 & 0 & -1 \\
            0 & 0 & 0 & 0 & 0 & 0 \\
            0 & 0 & 1 & 0 & 1 & 0 \\
            0 & 0 & 0 & 1 & 0 & 1 \\
            0 & 0 & 0 & 1 & 0 & 1 \\
            0 & 0 & 0 & 0 & 0 & 0 \\
          \end{array}
        \right),
\end{equation*}
and
\begin{equation*}
 B_{-1}= \left(
  \begin{array}{cccc}
    0 & 1 & 1 & 0 \\
    0 & 0 & 0 & 1 \\
    0 & -1 & -1 & 0 \\
    0 & 0 & 0 & -1 \\
    0 & 1 & 1 & 0 \\
    0 & 0 & 0 & 1 \\
  \end{array}
\right).
\end{equation*}
By step-by-step calculation, we can obtain the corresponding Khovanov homology presented in Table \ref{table:trefoil_homology}.
\begin{table}[htb!]
  \centering
  \caption{The Khovanov homology $H^{k,l}(L)$ of $L$.}\label{table:trefoil_homology}
  \begin{tabular}{c|c|c|c|c}
    \hline
    $H^{k,l}(L)$ & $k=0$ & $k=-1$ & $k=-2$ & $k=-3$ \\
    \hline
    $l=-1$ & $[v_{+}\otimes v_{+}]$ & 0 & 0 & 0 \\
    $l=-2$ & 0 & 0 & 0 & 0 \\
    $l=-3$ & $[v_{+}\otimes v_{-}]$ & 0 & 0 & 0 \\
    $l=-4$ & 0 & 0 & 0 & 0 \\
    $l=-5$ & 0 & 0 & $[v_{+}\otimes v_{-}-v_{-}\otimes v_{+}]$ & 0 \\
    $l=-6$ & 0 & 0 & 0 & 0 \\
    $l=-7$ & 0 & 0 & $[v_{-}\otimes v_{-}]_{2}$ & 0 \\
    $l=-8$ & 0 & 0 & 0 & 0 \\
    $l=-9$ & 0 & 0 & 0 & $[v_{-}\otimes v_{-}\otimes v_{-}]$ \\
    \hline
  \end{tabular}
\end{table}
Here, $k$ is the height and $l$ is the degree of the homology generators. The generator $[v_{-}\otimes v_{-}]_{2}$ exhibits a torsion of 2, meaning that $2[v_{-}\otimes v_{-}]_{2}=0$. The remaining generators are free. Thus we have
\begin{equation*}
  \begin{split}
    H^{-3}(L) &\cong \mathbb{K}, \\
    H^{-2}(L) &\cong \left\{
                       \begin{array}{ll}
                         \mathbb{K}\oplus \mathbb{K}, & \hbox{$\mathbb{K}$ is the field of characteristic 2;} \\
                         \mathbb{K}, & \hbox{otherwise.}
                       \end{array}
                     \right.\\
    H^{-1}(L) &=0,\\
    H^{0}(L) &\cong \mathbb{K}\oplus \mathbb{K}.
  \end{split}
\end{equation*}
Consider the case that $2$ is invertible in $\mathbb{K}$. The corresponding unnormalized Jones polynomial is given by
\begin{equation*}
  \hat{J}(L)=\mathcal{X}_{q}(L)=\sum\limits_{k}(-1)^{k}\qdim H^{k}(L)=q^{-1}+q^{-3}+q^{-5}-q^{-9}.
\end{equation*}
This coincides with the result shown in Example \ref{example:jones}.
\end{example}

\section{Evolutionary Khovanov homology}\label{section:evolutionary}

We encounter challenges in establishing a filtration process for links, to the extent that we lack even the concept of sublinks. In fact, morphisms in the category of links are provided by cobordisms, and cobordism constructions are geometric in nature. This presents a challenge in the application of links. Thus directly studying the filtration process on the category of links is not a favorable approach. Therefore, in order to obtain a persistent process for link versions, we consider establishing filtration from the perspective of Khovanov cochain complexes of links.

\subsection{Smoothing link}

Let $L$ be a link diagram. Let $x\in \mathcal{X}(L)$ be a crossing of $L$.  At crossing $x$, there are two smoothing options: the 0-smoothing denoted as $\rho_{0}(L,x)$ and the 1-smoothing denoted as $\rho_{1}(L,x)$. It is worth noting that $2^{\mathcal{X}(L)}=2^{\mathcal{X}(\rho_{0}(L,x))}\sqcup 2^{\mathcal{X}(\rho_{1}(L,x))}$. Thus the Khovanov chain groups of $\rho_{0}(L,x)$ and $\rho_{1}(L,x)$ are subspaces of the Khovanov chain group of $L$ without considering the gradings. Moreover, even when we consider gradings, the Khovanov complex $\mathcal{C}(\rho_{0}(L,x))$ or $\mathcal{C}(\rho_{1}(L,x))$ can still be a subcomplex of $\mathcal{C}(L)$ in certain cases.

When $x$ is a left-handed crossing. Assume that $n=|\mathcal{X}(L)|$ is the number of crossing of $L$. Each crossing in $\mathcal{X}(L)$ can be written of the form $(s_{1},s_{2},\dots,s_{n})$. Let $\lambda$ be the index of the crossing $x$ in $\mathcal{X}(L)$. We have a map $j_{0}:2^{\mathcal{X}(\rho_{0}(L,x))}\to 2^{\mathcal{X}(L)}$ given by
\begin{equation*}
  (s_{1},s_{2},\dots,s_{n-1})\to (s_{1},\dots,s_{\lambda-1},1,s_{\lambda},\dots,s_{n-1}).
\end{equation*}
Let $n_{-,0}$ be the number of left-handed crossings in $\mathcal{X}(\rho_{0}(L,x))$, and let $n_{+,0}$ be the number of right-handed crossings in $\mathcal{X}(\rho_{0}(L,x))$. It follows that
\begin{equation*}
  c(s)=c(j_{0}(s)),\quad n_{-,0}=n_{-}-1,\quad n_{+,0}=n_{+},\quad \ell(s)=\ell(j_{0}(s))-1.
\end{equation*}
Then we have an isomorphism of vector spaces
\begin{equation*}
  V^{\otimes c(s)}\{\ell(s)+n_{+}-2n_{-}\}\cong V^{\otimes c(j_{0}(s))}\{\ell(j_{0}(s))+n_{+,0}-2n_{-,0}\},
\end{equation*}
which is given by the degree shift. The degree difference is
\begin{equation*}
  \ell(j_{0}(s))+n_{+,0}-2n_{-,0}-\ell(s)-n_{+}+2n_{-}=1.
\end{equation*}
The height of both side are equal $\ell(s)-n_{-}=\ell(j_{0}(s))-n_{-,0}$. Thus the induced map
\begin{equation*}
  i_{0}:\mathcal{C}(\rho_{0}(L,x)) \to \mathcal{C}(L)
\end{equation*}
is an inclusion of degree -1 shift from the Khovanov complex $\mathcal{C}(\rho_{0}(L,x))$ to the Khovanov complex $\mathcal{C}(L)$. Moreover, one can verify $i_{0}d=di_{0}$ step by step by confirming $i_{0}d_{\xi}=d_{\xi}i_{0}$ at each stage. Hence, $\mathcal{C}(\rho_{0}(L,x))$ is the subcomplex of $\mathcal{C}(L)$.

When $x$ is a right-handed crossing. We can verify that $\mathcal{C}(\rho_{1}(L,x))$ is a subcomplex of $\mathcal{C}(L)$ using a similar approach as described above. Consider the map $j_{1}:2^{\mathcal{X}(\rho_{1}(L,x))}\to 2^{\mathcal{X}(L)}$ given by
\begin{equation*}
  (s_{1},s_{2},\dots,s_{n-1})\to (s_{1},\dots,s_{\lambda-1},0,s_{\lambda},\dots,s_{n-1}).
\end{equation*}
We can obtain a injection $i_{1}:\mathcal{C}(\rho_{1}(L,x)) \to \mathcal{C}(L)$ of degree 1 shift from the Khovanov complex $\mathcal{C}(\rho_{1}(L,x))$ to the Khovanov complex $\mathcal{C}(L)$. Thus, we have the following proposition.

\begin{proposition}\label{proposition:link}
Let $L$ be a link, and let $x$ be a crossing of $L$.
If $x$ is a left-handed crossing, $\mathcal{C}(\rho_{0}(L,x))$ is a subcomplex of $\mathcal{C}(L)$. If $x$ is a right-handed crossing, $\mathcal{C}(\rho_{1}(L,x))$ is a subcomplex of $\mathcal{C}(L)$.
\end{proposition}
The construction described above is called the \emph{smoothing link}, denoted by $\rho_{x}L$. Note that $\rho_{x} L=\rho_{0}(L,x)$ if $x$ is left-handed, and $\rho_{x} L=\rho_{1}(L,x)$ if $x$ is right-handed. By construction, we have the following result.
\begin{lemma}\label{lemma:smoothing}
Let $L$ be a link, and let $x,y$ be crossings of $L$. Then we have $\rho_{x}\rho_{y}L=\rho_{y}\rho_{x}L$.
\end{lemma}

In view of Lemma \ref{lemma:smoothing}, for a subset $S$ of $\mathcal{X}(L)$, we obtain a link $\rho_{S}L$ by applying the smoothing link step by step to crossings in $S$. Obviously, $\mathcal{C}(\rho_{S}(L,x))$ is the subcomplex of $\mathcal{C}(L)$.

\subsection{Evolutionary Khovanov homology}

A \emph{weighted link} is a link $L$ equipped with a function $f:\mathcal{X}(L)\to \mathbb{R}$ on the set of crossings of $L$. We arrange the crossings in $\mathcal{X}(L)$ in ascending order of their assigned values, denoted as $x_{1}, x_{2},\dots,x_{n}$. Then we have a filtration of links
\begin{equation*}
  L,\rho_{x_{1}}L,\rho_{x_{2}}\rho_{x_{1}}L,\dots,\rho_{x_{n}}\cdots\rho_{x_{2}}\rho_{x_{1}}L.
\end{equation*}
Note that the link $\rho_{x_{n}}\cdots\rho_{x_{2}}\rho_{x_{1}}L$ is an unknot, comprising a collection of disjoint circles. The filtration of links characterizes the process by which a complex link is gradually untangled, crossing by crossing, through smoothing. This process can be understood as the evolution of a link from complexity to simplicity.

For any real number $a$, we have the subset $\mathcal{X}(L,a)$ of $\mathcal{X}(L)$ consists of crossings $x$ such that $f(x)\leq a$. Then we have a link $\rho_{\mathcal{X}(L,a)}L$, which is called the \emph{$a$-indexed link}.

Let $(\mathbb{R},\leq)$ the category with real numbers as objects and pairs of form $a\leq b$ as morphisms.
\begin{theorem}\label{proposition:functor}
The construction $\mathcal{C}(\rho_{\mathcal{X}(L,-)}L)$ is a functor from the category $(\mathbb{R},\leq)^{\mathrm{op}}$ to the category of cochain complexes.
\end{theorem}
\begin{proof}
For any $a\leq b$, let $x_{t_{1}},\dots,x_{t_{u}}$ be the crossings in $\mathcal{X}(L,b)\backslash \mathcal{X}(L,a)$. By Proposition \ref{proposition:link} and Lemma \ref{lemma:smoothing}, the cochain complex $\mathcal{C}(\rho_{\mathcal{X}(L,b)}L)=\mathcal{C}(\rho_{t_{1}}\cdots \rho_{t_{u}}\rho_{\mathcal{X}(L,a)}L)$ is the subcomplex of $\mathcal{C}(\rho_{\mathcal{X}(L,a)}L)$. Let us denote $\theta_{a,b}:\mathcal{C}(\rho_{\mathcal{X}(L,b)}L)\to \mathcal{C}(\rho_{\mathcal{X}(L,a)}L)$. For real numbers $a\leq b\leq c$, we have the following commutative diagram.
\begin{equation*}
  \xymatrix{
  \mathcal{C}(\rho_{\mathcal{X}(L,c)}L)\ar@{->}[rr]^{\theta_{b,c}}\ar@{->}[rd]_{\theta_{a,c}}&&\mathcal{C}(\rho_{\mathcal{X}(L,b)}L)\ar@{->}[ld]^{\theta_{a,b}}\\
   &\mathcal{C}(\rho_{\mathcal{X}(L,a)}L)&
  }
\end{equation*}
It follows that $\theta_{a,b}\theta_{b,c}=\theta_{a,c}$. Note that $\theta_{a,a}=\mathrm{id}|_{\mathcal{C}(\rho_{\mathcal{X}(L,a)}L)}$ for any real number $a$. The desired result follows.
\end{proof}

For real numbers $a\leq b$, we have links $\rho_{\mathcal{X}(L,a)}L$ and $\rho_{\mathcal{X}(L,b)}L$. Note that there is an inclusion of Khovanov cochain complexes
\begin{equation*}
   \mathcal{C}(\rho_{\mathcal{X}(L,b)}L) \hookrightarrow \mathcal{C}(\rho_{\mathcal{X}(L,a)}L).
\end{equation*}
This induces the morphism of Khovanov homology
\begin{equation*}
   \lambda_{a,b}:H(\rho_{\mathcal{X}(L,b)}L) \to H(\rho_{\mathcal{X}(L,a)}L).
\end{equation*}
The \emph{$(a,b)$-evolutionary Khovanov homology} of the weighted link $(L,f)$ is defined by
\begin{equation*}
   H_{a,b}^{k}(L,f):=\im (H^{k}(\rho_{\mathcal{X}(L,b)}L) \to H^{k}(\rho_{\mathcal{X}(L,a)}L)),\quad k\geq 0.
\end{equation*}
\begin{remark}\label{remark:construction}
For a weighted link $(L,f)$ with crossings $x_{1}, x_{2},\dots,x_{n}$ of ascending weights, one can also obtain a filtration of links
\begin{equation*}
  L,\rho_{x_{n}}L,\rho_{x_{n-1}}\rho_{x_{n}}L,\dots,\rho_{x_{1}}\cdots\rho_{x_{n-1}}\rho_{x_{n}}L.
\end{equation*}
For any real number $a$, let $\mathcal{X}_{a}(L)$ be the set of crossing with weight $f(x)\geq a$. Then the construction $\mathcal{C}(\rho_{\mathcal{X}_{-}(L)}L)$ is a functor from the category $(\mathbb{R},\leq)$ to the category of cochain complexes. For real numbers $a\leq b$, we define the \emph{$(a,b)$-evolutionary Khovanov homology} of the weighted link $(L,f)$ as
\begin{equation*}
   H_{a,b}^{k}(L,f):=\im (H^{k}(\rho_{\mathcal{X}_{a}(L)}L) \to H^{k}(\rho_{\mathcal{X}_{b}(L)}L)),\quad k\geq 0.
\end{equation*}
This definition shares the same fundamental idea as the previous definition.
\end{remark}

The rank of $H_{a,b}^{k}(L,f)$ is called the \emph{$(a,b)$-evolutionary Betti number}, denoted by $\beta_{a,b}(L,f)$, which is the crucial features for us to conduct data analysis. In particular, if we take $a=b$, we have that $H_{a,b}^{k}(L,f)=H^{k}(\rho_{\mathcal{X}(L,a)}L)$. Furthermore, we can define the \emph{$(a,b)$-evolutionary unnormalized Jones polynomial} as
\begin{equation*}
  \hat{J}_{a,b}(L)=\sum\limits_{k}(-1)^{k}\qdim H^{k}_{a,b}(L).
\end{equation*}

As a direct corollary of Proposition \ref{proposition:functor}, we have the following result, which shows that the evolutionary Khovanov homology is a (co)persistence module \cite{bi2022cayley}.
\begin{theorem}
The evolutionary Khovanov homology $H:(\mathbb{R},\leq)^{\mathrm{op}}\to \mathrm{Vec}_{\mathbb{K}}$ is a functor from the category $(\mathbb{R},\leq)^{\mathrm{op}}$ to the category of $\mathbb{K}$-module.
\end{theorem}
Evolutionary Khovanov homology tracks how the generators of Khovanov homology evolve with changes in parameter filtration. This concept shares a remarkable similarity with persistent homology. Yet, there are fundamental distinctions between the evolution process of evolutionary Khovanov homology and the persistence process of persistent homology: the former relies on smoothing the link, while the latter is established through the Vietoris-Rips complex, ensuring a continuous persistence.

\begin{example}\label{example:unknot}
Consider the link $L$ in Figure \ref{figure:evolution_knot}. Link $L$ has four crossings, labeled $x_{1},x_{2},x_{3},x_{4}$ in the figure.
We consider the weighted functions $f,g:\mathcal{X}(L)\to \mathbb{R}$ defined by
\begin{equation*}
  f(x_{1})=1,f(x_{2})=2,f(x_{3})=3,f(x_{4})=5,
\end{equation*}
and
\begin{equation*}
  g(x_{1})=1,g(x_{2})=3,g(x_{3})=2,g(x_{4})=4.
\end{equation*}
This gives us the following filtrations of links:
\begin{equation*}
  L,\rho_{x_{1}}L,\rho_{x_{2}}\rho_{x_{1}}L,\rho_{x_{3}}\rho_{x_{2}}\rho_{x_{1}}L,\rho_{x_{4}}\rho_{x_{3}}\rho_{x_{2}}\rho_{x_{1}}L
\end{equation*}
and
\begin{equation*}
  L,\rho_{x_{1}}L,\rho_{x_{3}}\rho_{x_{1}}L,\rho_{x_{2}}\rho_{x_{3}}\rho_{x_{1}}L,\rho_{x_{4}}\rho_{x_{2}}\rho_{x_{3}}\rho_{x_{1}}L.
\end{equation*}
\begin{figure}[htb!]
	\centering
	\includegraphics[width=0.8\textwidth]{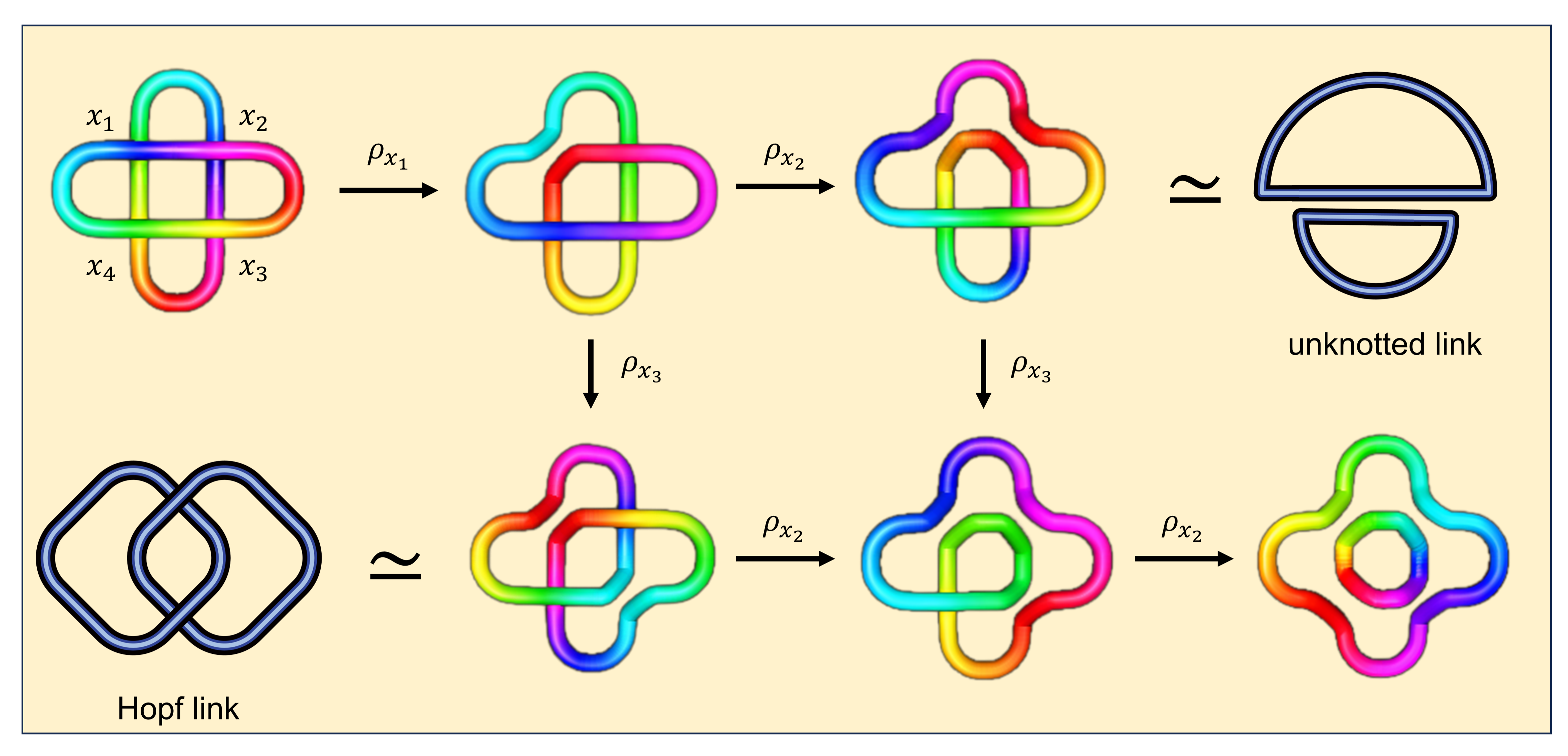}\\
	\caption{Link $L$ produces different filtrations of links when processed through the crossings $x_{1},x_{2},x_{3}$ and through the crossings $x_{1},x_{3},x_{2}$.}\label{figure:evolution_knot}
\end{figure}
Note that link $L$ is unknotted, so its Khovanov homology is trivial. The links in the filtration given by the weighted function $f$ are all unknotted links, hence their corresponding evolutionary Khovanov homologies are also trivial. On the other hand, note that the link $\rho_{x_{3}}\rho_{x_{1}}L$ is a Hopf link. Its Khovanov homology has four generators, and the Khovanov homology is given by
\begin{equation*}
  \begin{split}
    H^{-2}(\rho_{x_{3}}\rho_{x_{1}}L) &\cong \mathbb{K}\oplus \mathbb{K},\\
    H^{-1}(\rho_{x_{3}}\rho_{x_{1}}L) &=0,\\
    H^{0}(\rho_{x_{3}}\rho_{x_{1}}L) &\cong \mathbb{K}\oplus \mathbb{K}.
  \end{split}
\end{equation*}
The evolutionary Khovanov homology $H^{\ast}_{2,2}(L,g)$ is non-trivial.
This example illustrates that even if an unknotted link has trivial Khovanov homology, its evolutionary Khovanov homology may not be trivial. Moreover, different choices of weighting functions can produce different filtrations of links, leading to variations in their evolutionary Khovanov homology.
\end{example}

\subsection{Representations of evolutionary features}

In the previous section, we proved that evolutionary Khovanov homology is a functor. Consequently, evolutionary Khovanov homology also has representations similar to barcode and persistence diagram in persistent homology theory.

Given a weighted link $(L,f)$, since the links we consider have a finite number of crossings, we can arrange the crossings of the link $L$ in ascending order of their weights as $x_{1},x_{2},\dots,x_{n}$. For any integers $1\leq i\leq j\leq n$, we obtain an evolutionary Khovanov homology $H_{f(x_{i}),f(x_{j})}^{k}(L,f)$. Let $\mathbf{H}=\bigoplus\limits_{i}H_{f(x_{i})}(L,f)$, and let $t:\mathbf{H}\to \mathbf{H}$ be given by the map $\lambda_{f(x_{i}),f(x_{i+1})}:H_{f(x_{i+1})}(L,f)\to H_{f(x_{i})}(L,f)$. Then, for any element $g$ in the polynomial ring $\mathbb{K}[t]$, we obtain a map
\begin{equation*}
  g:\mathbf{H}\to \mathbf{H}.
\end{equation*}
This implies that $\mathbf{H}$ is a finitely generated $\mathbb{K}[t]$-module. By the decomposition theorem for finitely generated modules over a principal ideal domain, we have
\begin{theorem}
Let $(L,f)$ be a weighted link. We have a decomposition of the evolutionary Khovanov homolog of $(L,f)$ given by
\begin{equation}\label{equation:decomposition}
  \mathbf{H}\cong \bigoplus_{k}t^{b_{k}}\cdot \mathbb{K}[t] \oplus \left( \bigoplus\limits_{l}t^{c_{l}}\cdot \frac{\mathbb{K}[t]}{t^{d_{l}}\cdot \mathbb{K}[t]}\right).
\end{equation}
\end{theorem}
In the decomposition mentioned above, the $\mathbb{K}[t]$-module $\mathbf{H}$ has two components: the free part and the torsion part. For the free part, $b_{k}$ represents a generator of the evolutionary Khovanov homology, which has weight 1 until smoothing at crossing $x_{b_{k}}$ and becomes weight 0 after smoothing at crossing $x_{b_{k}}$. For the torsion part, $c_{l}$ represents a generator that, after smoothing at crossing $x_{c_{l}}$, its weight becomes 0. Before smoothing at crossing $x_{c_{l}}$, this generator has weight 1 after smoothing at crossing $x_{c_{l}-d_{l}}$ and weight 0 before smoothing at crossing $x_{c_{l}-d_{l}}$.

Evolutionary Khovanov homology reflects the changes in homological generators of a link as it undergoes smoothing. This provides a more nuanced characterization of the topological features of the link. It also implies that the characteristic representation of evolutionary Khovanov homology is highly valuable in application. Common representations include barcode and persistence diagrams. Considering the decomposition of evolutionary Khovanov homology, each generator's information can be represented using intervals. For the decomposition \eqref{equation:decomposition}, the generators of the free part can be represented by intervals $(-\infty,b_{k}]$, while for the torsion part, their generators can be represented by intervals $[c_{l}-d_{l},c_{l}]$. This collection of intervals provides the barcode of evolutionary Khovanov homology. Another well-known representation is the persistence diagram. For the generators of the free part, they are represented by pairs of the form $(-\infty,b_{k})$, while for the torsion part, pairs of the form $(c_{l}-d_{l},c_{l})$ are used. These pairs correspond to points on the plane $\mathbb{R}^{2}$, and these discrete points provide the persistence diagram representation of evolutionary Khovanov homology. Other tools such as Betti curves and persistence landscapes are commonly used for representing and analyzing topological features. We will demonstrate these representations in examples and applications.

\begin{example}
Consider the weighted trefoil knot $(L,f)$ with $f:\mathcal{X}(L)\to \mathbb{R}$ defined as $f(x_{1})=1$, $f_{x_{2}}=2$, and $f_{x_{3}}=3$. Then we have a filtration of links $L,\rho_{x_{1}}L,\rho_{x_{2}}\rho_{x_{1}}L,\rho_{x_{3}}\rho_{x_{2}}\rho_{x_{1}}L,$ shown in Figure \ref{figure:trefoil_barcode}\textbf{a}. This filtration illustrates the process of untangling a crossing of a trefoil by smoothing.
\begin{figure}[htb!]
	\centering
	\includegraphics[width=0.9\textwidth]{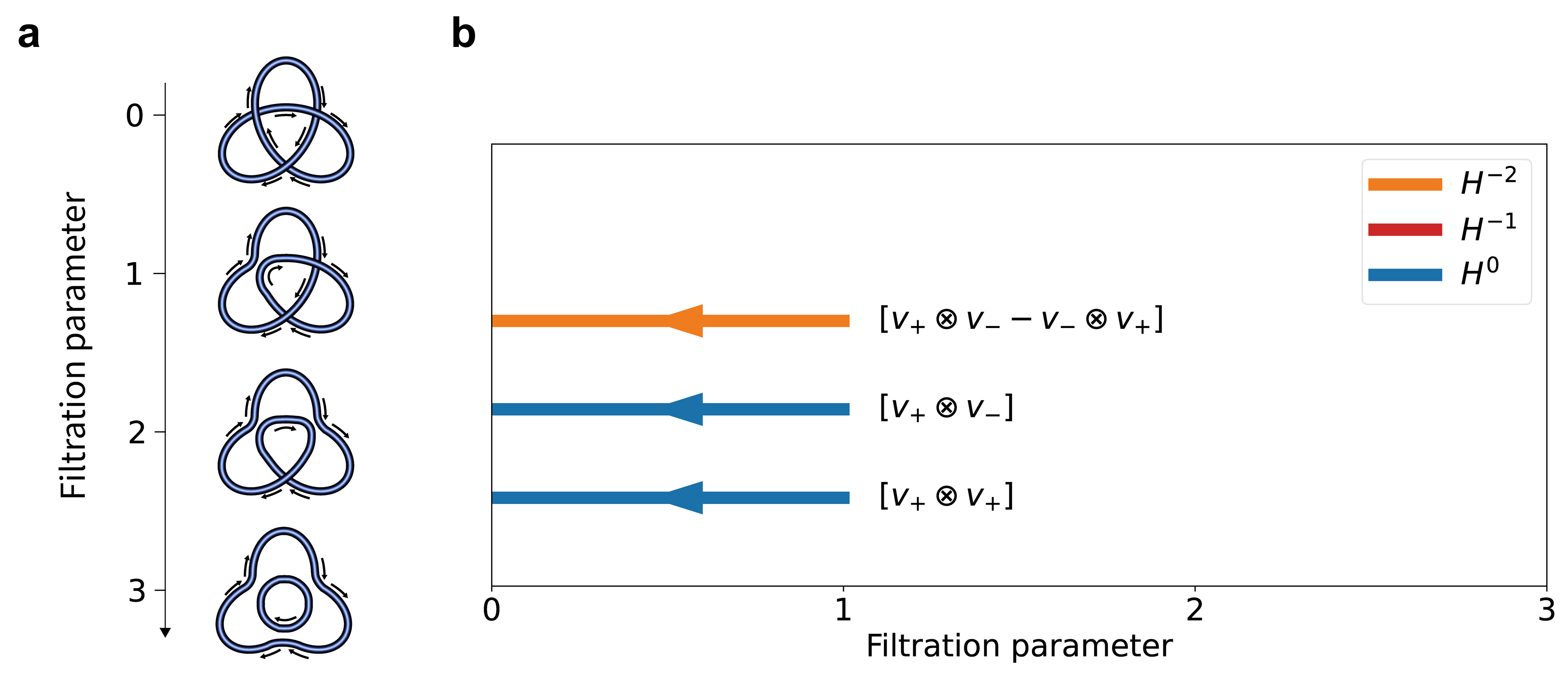}\\
	\caption{\textbf{a} The filtration of smoothing links of the weighted trefoil link $(L,f)$. \textbf{b} The barcode of the evolutionary Khovanov homology of $(L,f)$.}\label{figure:trefoil_barcode}
\end{figure}
Note that the last two links are both unknotted, so they have trivial Khovanov homology. Now, let's first examine the Khovanov complex of the link $\rho_{x_{1}}L$. Note that the map $i_{0}:2^{\mathcal{X}(\rho_{x_{2}}L)}\to 2^{\mathcal{X}(L)}$ is given by $(s_{1},s_{2})\to (1,s_{1},s_{2})$. Hence, we can verify the commutative diagram between the Khovanov complex of $\rho_{x_{1}}L$ and the Khovanov complex of $L$.
\begin{equation*}
  \xymatrix{
  &0\ar@{->}[r]\ar@{->}[d]&V\otimes V \ar@{->}[r]^{d^{-2}}\ar@{^{(}->}[d]& V\oplus V \ar@{->}[r]^{d^{-1}}\ar@{^{(}->}[d]&V\otimes V \ar@{->}[r]\ar@{^{(}->}[d]& 0\\
   0\ar@{->}[r]&V\otimes V \otimes V\ar@{->}[r]^{d^{-3}}&\bigoplus\limits_{i=1}^{3} V\otimes V \ar@{->}[r]^{d^{-2}}& V\oplus V \oplus V \ar@{->}[r]^{d^{-1}}&V\otimes V \ar@{->}[r]& 0
 }
\end{equation*}
We select the basis of $V\otimes V$ as $v_{+}\otimes v_{+}$, $v_{+}\otimes v_{+}$, $v_{+}\otimes v_{+}$, $v_{+}\otimes v_{+}$, and for $V\oplus V$, the basis is chosen as $(v_{+},0)$, $(v_{-},0)$, $(0,v_{+})$, $(0,v_{-})$. Then, the left representation matrices of the differentials $d^{-2}$ and $d^{-1}$ in the Khovanov complex $\mathcal{C}^{\ast}(\rho_{x_{1}}L)$ are as follows:
\begin{equation*}
  B_{-2}=\left(
           \begin{array}{cccc}
             1 & 0 & 1 & 0 \\
             0 & 1 & 0 & 1 \\
             0 & 1 & 0 & 1 \\
             0 & 0 & 0 & 0 \\
           \end{array}
         \right),\quad B_{-1}=\left(
                                \begin{array}{cccc}
                                  0 & 1 & 1 & 0 \\
                                  0 & 0 & 0 & 1 \\
                                  0 & -1 & -1 & 0 \\
                                  0 & 0 & 0 & -1 \\
                                \end{array}
                              \right).
\end{equation*}
From matrix calculations, we can obtain the generators of the Khovanov homology of $\rho_{x_{1}}L$ as in Table \ref{table:smoothing_homology}.
\begin{table}[htb!]
  \centering
  \caption{The Khovanov homology $H^{k,l}(\rho_{x_{1}}L)$ of $\rho_{x_{1}}L$.}\label{table:smoothing_homology}
  \begin{tabular}{c|c|c|c}
    \hline
    $H^{k,l}(\rho_{x_{1}}L)$ & $k=0$ & $k=-1$ & $k=-2$ \\
    $l=0$ & $[v_{+}\otimes v_{+}]$ & 0 & 0 \\
    $l=-1$ & 0 & 0 & 0 \\
    $l=-2$ & $[v_{+}\otimes v_{-}]$ & 0 & 0 \\
    $l=-3$ & 0 & 0 & 0 \\
    $l=-4$ & 0 & 0 & $[v_{+}\otimes v_{-}-v_{-}\otimes v_{+}]$ \\
    $l=-5$ & 0 & 0 & 0 \\
    $l=-6$ & 0 & 0 & $[v_{-}\otimes v_{-}]$ \\
    \hline
  \end{tabular}
\end{table}
Therefore, the Khovanov homology of $\rho_{x_{1}}L$ is given by
\begin{equation*}
  \begin{split}
    H^{-2}(\rho_{x_{1}}L) &\cong \mathbb{K}\oplus \mathbb{K},\\
    H^{-1}(\rho_{x_{1}}L) &=0,\\
    H^{0}(\rho_{x_{1}}L) &\cong \mathbb{K}\oplus \mathbb{K}.
  \end{split}
\end{equation*}
The corresponding unnormalized Jones polynomial is given by
\begin{equation*}
  \hat{J}(L)=\mathcal{X}_{q}(L)=\sum\limits_{k}(-1)^{k}\qdim H^{k}(L)=1+q^{-2}+q^{-4}+q^{-6}.
\end{equation*}
Comparing Table \ref{table:trefoil_homology} and Table \ref{table:smoothing_homology}, we observe that the homology generators $[v_{+}\otimes v_{+}]$, $[v_{+}\otimes v_{-}]$, and $[v_{+}\otimes v_{-}-v_{-}\otimes v_{+}]$ of $H^{\ast}(\rho_{x_{1}}L)$ are mapped to generators in $H^{\ast}(L)$.
And the generator $[v_{-}\otimes v_{-}]$ maps to the torsion part in $H^{\ast}(L)$. Assuming that $2$ is invertible in $\mathbb{K}$, we can conclude that the generator $[v_{-}\otimes v_{-}]$ vanishes in $H^{\ast}(L)$. The corresponding barcode of the evolutionary Khovanov homology is shown in Figure \ref{figure:trefoil_barcode}\textbf{b}. There are three bars, representing the generators $[v_{+}\otimes v_{+}]$, $[v_{+}\otimes v_{-}]$, and $[v_{+}\otimes v_{-}-v_{-}\otimes v_{+}]$. The arrows indicate that the cohomology generators emerge from later moments and persist towards earlier moments.
These generators can be represented by intervals as $[0,1]$, $[0,1]$, and $[0,1]$, respectively, each with degrees $-1$, $-3$, and $-5$. Besides, the $(0,1)$-evolutionary unnormalized Jones polynomial of $(L,f)$ is
\begin{equation*}
  \hat{J}_{0,1}(L)=\sum\limits_{k}(-1)^{k}\qdim H^{k}_{0,1}(L)=q^{-1}+q^{-3}+q^{-5}.
\end{equation*}
\end{example}

\subsection{Distance-based filtration of links}

Traditional approaches to studying knots or links primarily focus on their topological properties. However, considering knots and links as objects within a metric space, their geometric properties are equally significant. In this section, we study the geometric information and topological characteristics of links by exploring distance-based filtration. This method allows us to extract richer and more effective information about links.

Consider a link $L$ with crossings projected into a space $\mathbb{R}^{2}$. Let $\mathcal{X}(L)$ be the set of crossings. We have a function $f:\mathcal{X}(L)\to \mathbb{R}$ defined as follows. For a crossing $x\in \mathbb{R}^{2}$, we can construct a disk $D(x, r)$ with center $x$ and radius $r$. Then, $f(x)$ is defined as the maximal real number $r$ such that there are no other crossings within the interior of $D(x, r)$ apart from $x$. Mathematically, we have
\begin{equation}\label{equation:weight}
  f(x)=\max \{r| d(x,y)\geq r \text{ for any crossing $y\neq x$ in $\mathcal{X}(L)$}\}.
\end{equation}
Geometrically, we connect points that are within a distance $r$. When $r<f(x)$, the point $x$ remains isolated. Based on this construction, we obtain a weighted link $(L,f)$. Using the method described in \ref{remark:construction}, we can obtain a filtration of links, which we refer to as the \emph{distance-based filtration of links}. In the above construction, we can metaphorically say that we smooth out the isolated crossings first, gradually breaking down the entire knot step by step.

Now, For real numbers $a\leq b$, the $(a,b)$-evolutionary Khovanov homology the link $L$ is
\begin{equation*}
   H_{a,b}^{k}(L):=\im (H^{k}(\rho_{\mathcal{X}_{a}(L)}L) \to H^{k}(\rho_{\mathcal{X}_{b}(L)}L)),\quad k\geq 0.
\end{equation*}
Specifically, when $a$ and $b$ are sufficiently large, $H_{a,b}^{k}(L)=H^{k}(L)$. Conversely, when $a$ and $b$ are sufficiently small, we have $H_{a,b}^{k}(L)=0$. We will illustrate this method with an example.

\begin{example}
Consider the link $L$ embedded in $\mathbb{R}^{3}$ shown in Figure \ref{geometric_link}\textbf{a}. This is a knot of $7_{6}$ type.
\begin{figure}[htb!]
  \centering
  \includegraphics[width=0.5\textwidth]{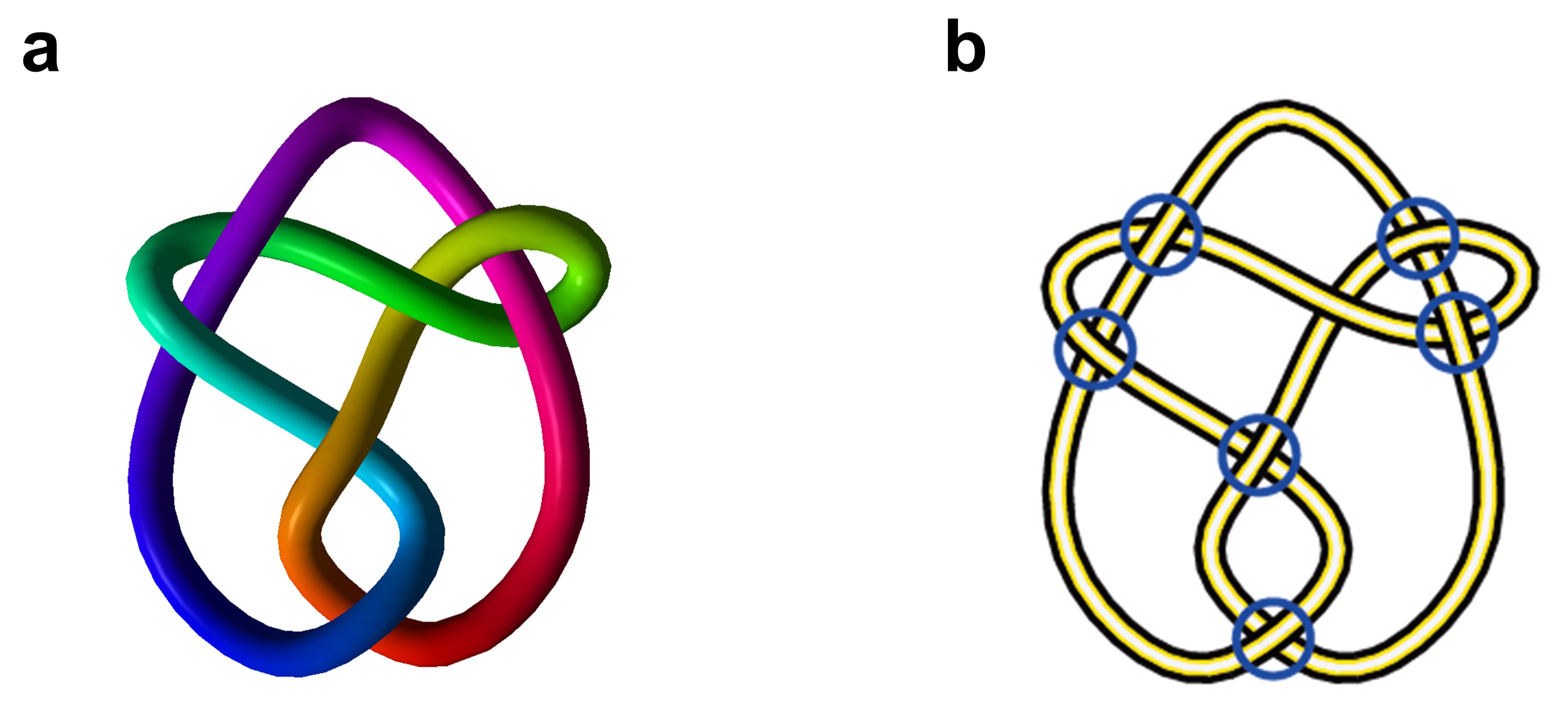}\\
  \caption{\textbf{a} A knot $L$ of type $7_{6}$ in 3-dimensional space. \textbf{b} The corresponding knot diagram of $L$.}\label{geometric_link}
\end{figure}
The coordinates of these crossings are given below:
\begin{equation*}
\begin{split}
& (-3.68122, 2.1618, 0.520849),(-2.31313, 4.52637, -0.526226), \\
& (-0.291898, -0.0329635, 0.5289),(-0.000160251, -3.82999, -0.657526), \\
& (1.29451, 3.02755, -0.309725), (2.99467, 4.45183, 0.450002), \\
& (3.79753, 2.50471, -0.482759).
\end{split}
\end{equation*}
We project the knot onto the $xOy$ plane, obtaining a knot diagram as shown in Figure \ref{geometric_link}\textbf{b}.
\begin{figure}[htb!]
  \centering
  \includegraphics[width=0.75\textwidth]{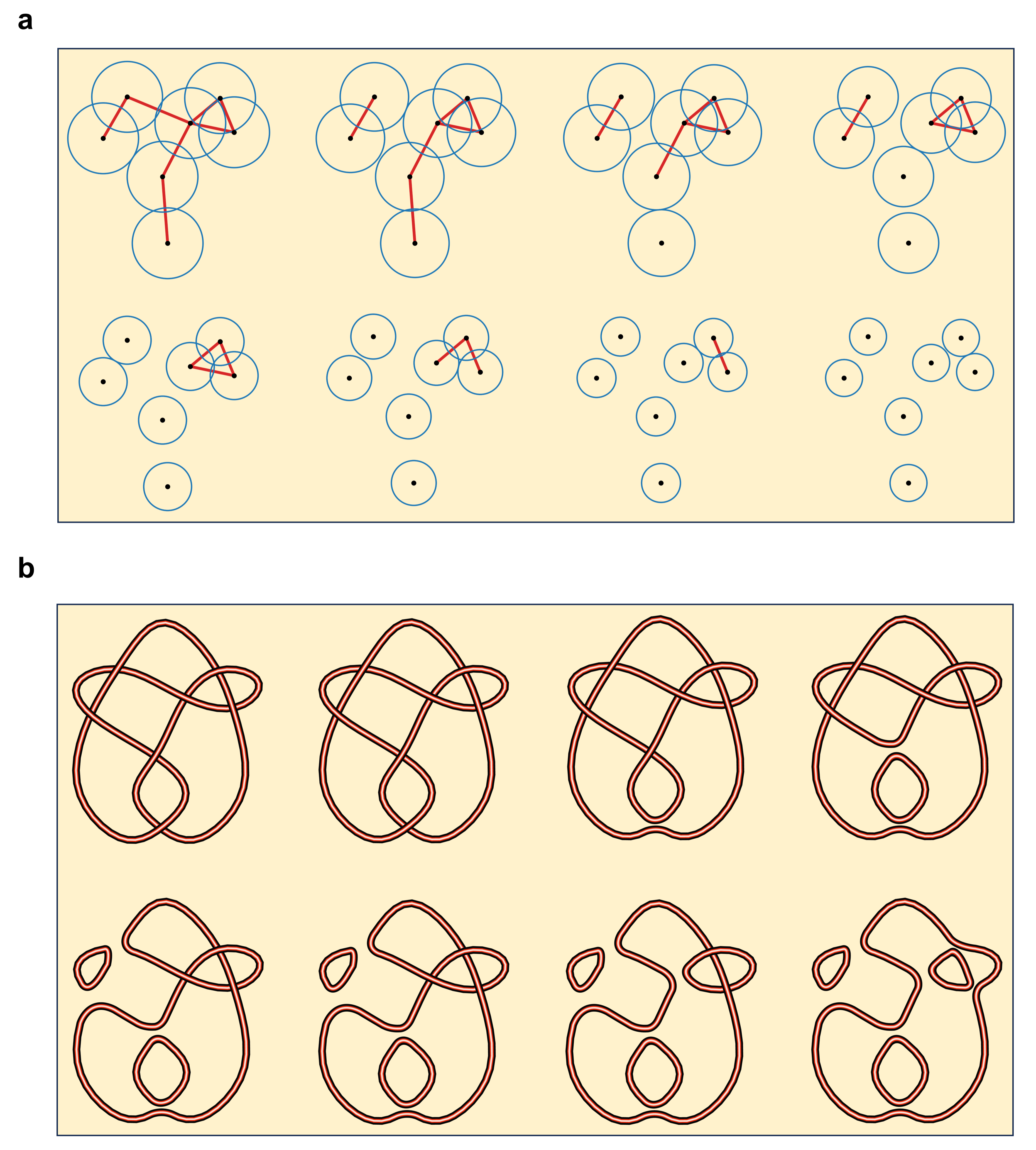}\\
  \caption{\textbf{a} As the distance decreases, isolated crossing points undergo gradual smoothing. \textbf{b} The filtration of links provided by the distance-based weighted function.}\label{figure:distance_filtration}
\end{figure}
Through the construction of the weighted function in Equation \eqref{equation:weight}, we can obtain a weighted link $(L,f)$. Figure \ref{figure:distance_filtration}\textbf{b} depicts the process of assigning weights to crossings. Subsequently, we can derive a filtration of links as illustrated in Figure \ref{figure:distance_filtration}\textbf{b}.
The variations in Figure \ref{figure:distance_filtration}\textbf{a} correspond to eight different cases, each yielding a distinct result. In Table \ref{table:distance_filtration}, we describe the different critical distances corresponding to the changes in Figure \ref{figure:distance_filtration}\textbf{a}, along with their respective link types. Here, $7_{6}$ and $3_{1}$ represent types in the knot table. Specifically, $3_{1}$ denotes the trefoil. The links $5^{2}_{1}$ and $2^{2}_{1}$ are representations in Rolfsen's Table of Links, where $5^{2}_{1}$ is the Whitehead link and $2^{2}_{1}$ is the Hopf link. Additionally, $n\bigcirc$ denotes $n$ separate unknots $\bigcirc$.

\begin{table}[htb!]
  \centering
  \caption{The link types of the filtration of links}\label{table:distance_filtration}
  \begin{tabular}{c|c|c|c|c|c|c|c|c}
    \hline
        Filtration & 1 & 2 & 3 & 4 & 5 & 6 & 7 & 8 \\
        \hline
    Critical distance & 2.019 & 1.953 & 1.904 & 1.724 & 1.366 & 1.279 & 1.109 & 1.053 \\
    Type of links & $7_{6}$  & $5^{2}_{1}$ & $5^{2}_{1}$+$\bigcirc$  &  $5^{2}_{1}$+ $\bigcirc$ & $3_{1}$+2$\bigcirc$ & $3_{1}$+2$\bigcirc$ &  $2^{2}_{1}$+2$\bigcirc$ & 4 $\bigcirc$ \\
    \hline
  \end{tabular}
\end{table}
Furthermore, for each filtration distance, we can obtain the corresponding Khovanov homology. Figure \ref{figure:poincare_figures} illustrates the evolution of the graded Poincar\'{e} polynomial of Khovanov homology. The $x$-axis represents the filtration distance, while the $y$-axis denotes the Euler characteristic $\chi_{1}=\chi_{1}(L_{r})$ for the link $L_{r}$ at distance $r$. Each subfigure in Figure \ref{figure:poincare_figures} represents the surface of the graded Poincar\'{e} polynomial of the Khovanov homology $H^{\ast}(L_{r})$.
\begin{figure}[htb!]
  \centering
  \includegraphics[width=0.75\textwidth]{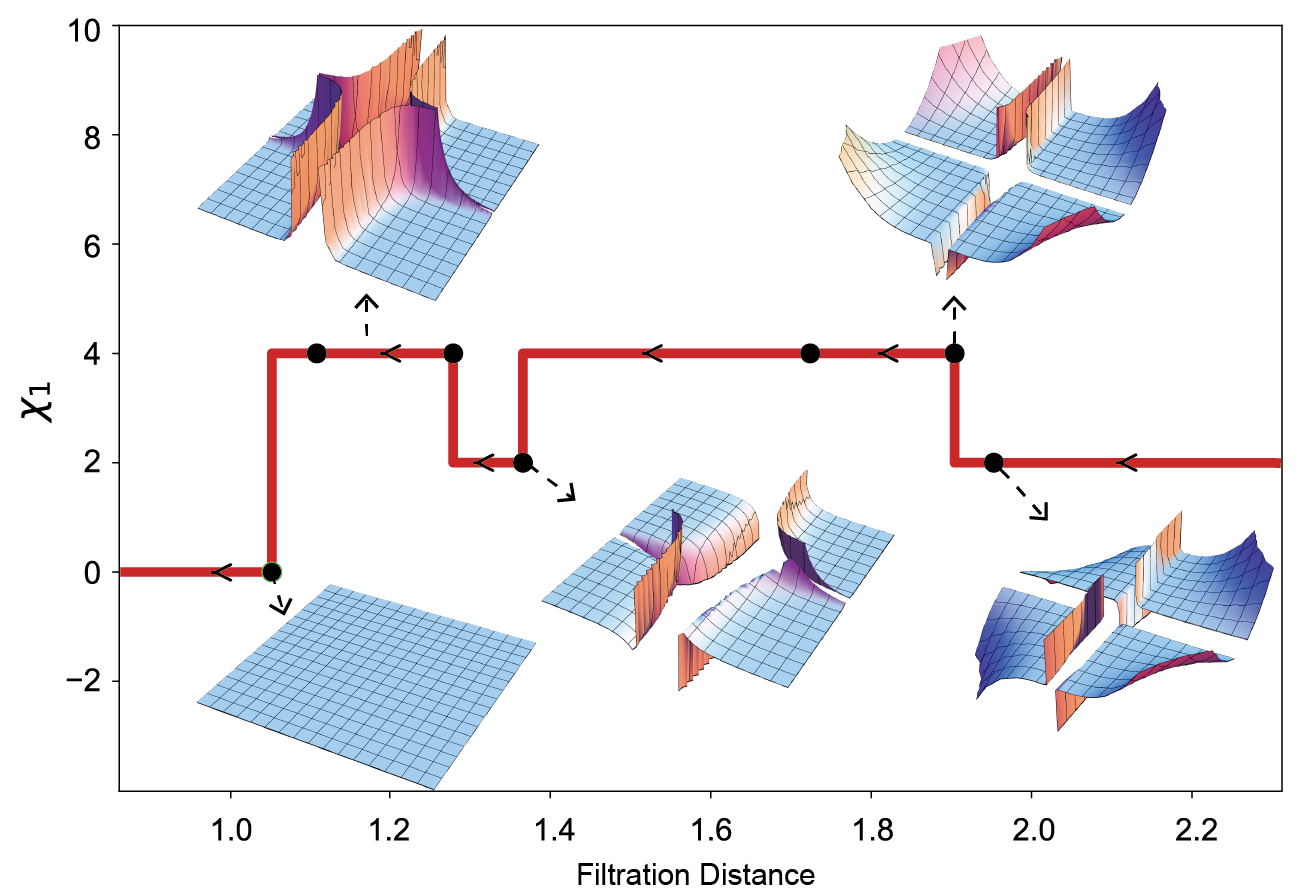}\\
  \caption{The representation of evolutionary Khovanov homology. Each subfigure represents the surface of the graded Poincar\'{e} polynomial of the Khovanov homology at the corresponding distance parameter. The $y$-axis denotes the value of Euler characteristic $\chi_{q}$ for the case $q=1$.}\label{figure:poincare_figures}
\end{figure}
The graded dimensions of the Khovanov homology of the links are the graded Betti numbers parameterized by $q$. When we set $q=1$, it reduces to the usual Betti numbers, representing the number of generators. In persistent homology theory, for a given dimension $k$ and distance $r$, the Betti number $\beta_{k}$ is a real number. In evolutionary Khovanov homology, for a given dimension $k$ and distance $r$, the graded Betti number $\beta_{k}(q)$ is a polynomial in $q$. In other words, the graded Betti number not only includes information about the number of generators but also about the degree of each generator. In Table \ref{table:gradedBetti}, we can find the observe the evolution of the graded Betti numbers in evolutionary Khovanov homology for different $k$.

\begin{table}[htb!]
  \centering
    \caption{The graded Betti of the filtration of links}\label{table:gradedBetti}
\begin{tabular}{|c|c|c|c|c|c|}
    \hline
    & \multicolumn{5}{|c|}{Distance} \\
    \cline{2-6}
    Degree & 1.053 & 1.109 & 1.279--1.366 & 1.724--1.953 & 2.019 \\
    \hline
    $k\geq 1$ & 0 & 0 & 0 & $q^{4}+1$ & $q^{3}+q+q^{-1}$ \\
    $k=0$ & 0 & $1+q^{-2}$ & $q^{-1}+q^{-3}$ & $2+2q^{-2}$ & $2q^{-1}+2q^{-3}$ \\
    $k=-1$ & 0 & 0 & 0 & $q^{-2}$ & $2q^{-3}+q^{-5}$ \\
    $k=-2$ & 0 & $q^{-4}+q^{-6}$ & $q^{-5}$ & $q^{-4}+q^{-6}$ & $2q^{-5}+2q^{-7}$ \\
    $k\leq -3$ & 0 & 0 & $q^{-9}$ & $q^{-8}$ & $q^{-7}+3q^{-9}+q^{-11}+q^{-3}$ \\
    \hline
\end{tabular}
\end{table}

\end{example}

\section{Applications}\label{section:application}

Pseudoknots in biology are a widely occurring motif in nucleic acids that play vital roles in various biological processes \cite{pleij1990pseudoknots,staple2005pseudoknots}. Comprising at least two intertwined stem-loop structures, pseudoknots are characterized by the intercalation of one stem's half into the halves of another stem, creating a knot-like three-dimensional shape. Despite their complex appearance, these structures can essentially be simplified to an unknot through deintercalation. Consequently, traditional knot theory-based data analysis methods may not be applicable in these scenarios.

Pseudoknots may be mathematically trivial, but they are biologically significant. Pseudoknots are involved in several critical cellular functions, particularly those involving RNAs, such as ribosomal frameshifting \cite{jones2022crystal} and regulation of gene expression \cite{pleij1990pseudoknots}. Pseudoknots also contribute to the catalytic activity of ribozymes, where they help stabilize the active site and facilitate chemical reactions. Additionally, they are essential in viral replication processes and impact RNA stability and folding, which affects RNA's overall functionality and interactions within the cell. This complex interplay of structure and function exemplifies the importance of pseudoknots in molecular biology, offering potential targets for therapeutic interventions and advancing our understanding of genetic regulation and viral mechanisms \cite{kim2008solution}. Therefore, a successful data analysis method designed to study a RNA pseudoknot structure must be able to accurately detect and analysis the structural differences of a regular RNA structure and keenly capture the pseudoknot's stem intercalation.

\begin{figure}[htb!]
	\centering
	\includegraphics[width=0.75\textwidth]{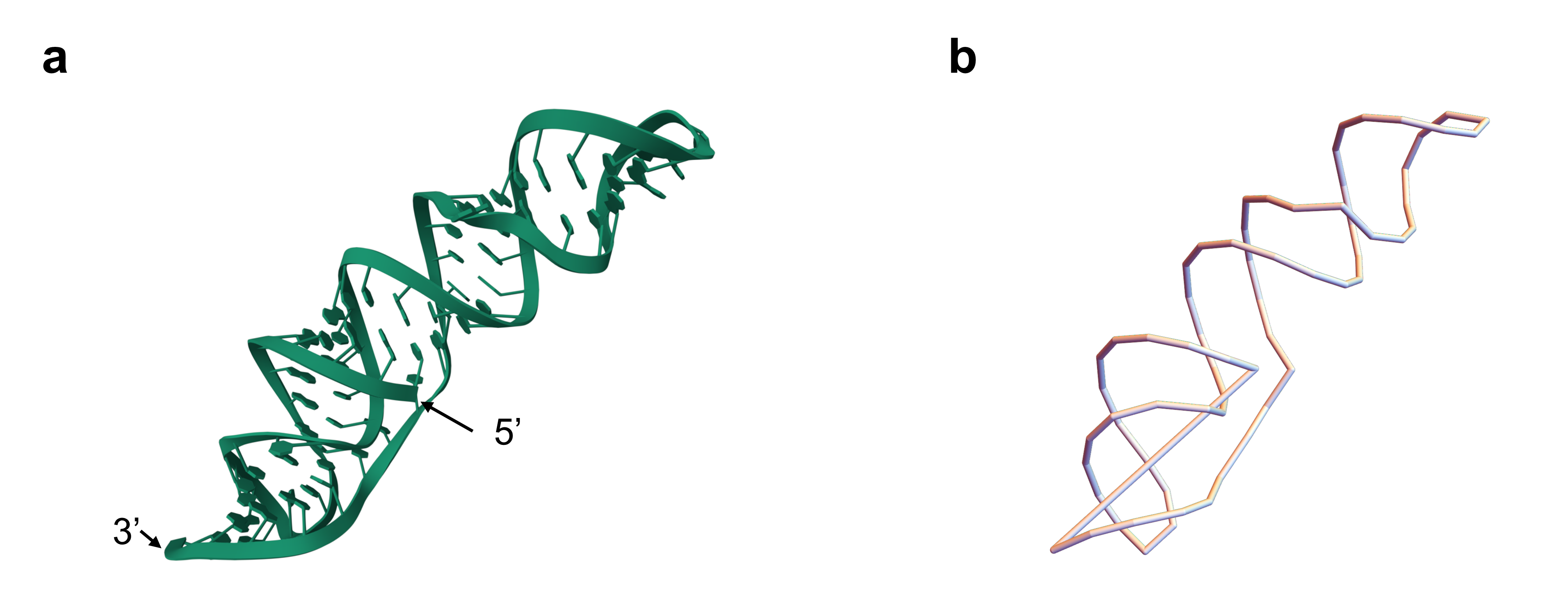}\\
	\caption{\textbf{a} The representation of the SARS-CoV-2 frameshifting pseudoknot with the 5' and 3' ends. \textbf{b} The corresponding abstract knot of the SARS-CoV-2 frameshifting pseudoknot formed by connecting the two ends.}\label{figure:knots_protein}
\end{figure}

\begin{figure}[htb!]
	\centering
	\includegraphics[width=0.9\textwidth]{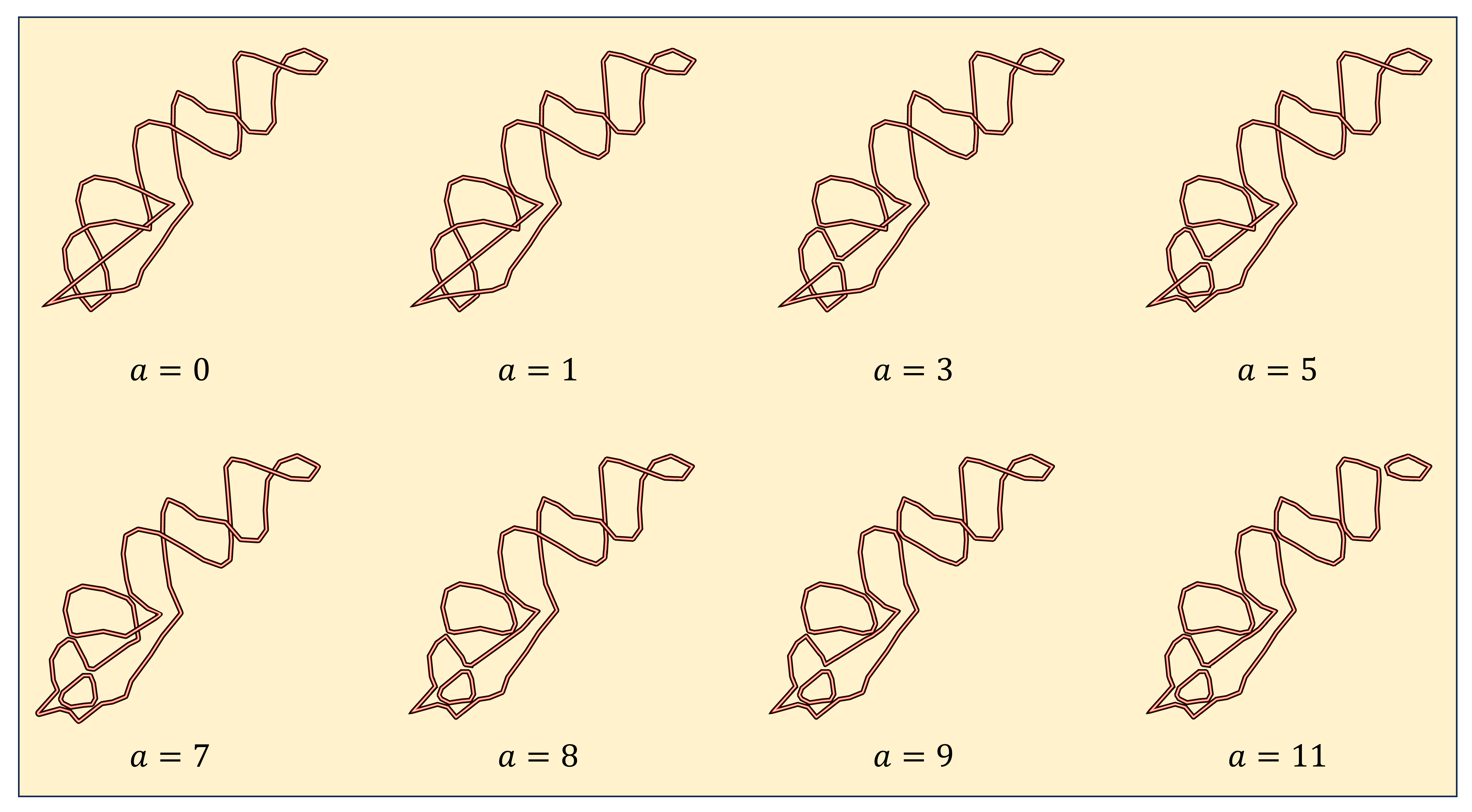}\\
	\caption{The filtration of smoothing links of the corresponding knot diagram of the SARS-CoV-2 frameshifting pseudoknot.}\label{figure:knots_protein2}
\end{figure}

In this study, we employed evolutionary Khovanov homology to investigate the structure of the SARS-CoV-2 frameshifting pseudoknot (PDB ID: 7LYJ), focusing on its RNA backbone to elucidate potential functional implications. Initially, we simplified the molecular structure by representing each RNA residue solely by its phosphorus atom, and connecting these atoms with linear segments to form a continuous backbone, directed from the 5' to 3' end, see Figure \ref{figure:knots_protein}\textbf{a}. This abstraction was followed by transforming the linear RNA chain into a closed loop, ensuring continuity by connecting the terminal phosphorus atoms. Such closure is essential for applying knot theory, as it converts the molecular structure into a topologically relevant form as in Figure \ref{figure:knots_protein}\textbf{b}. Lastly, to facilitate the analysis of the RNA's topological properties, we projected the closed-loop structure onto the xz plane, generating a knot diagram. Along the numbering of crossings, the value of the weight function corresponds to the number assigned to each crossing. Consequently, we obtain a filtration of links, as shown in Figure \ref{figure:knots_protein2}.

Using the method described in Section \ref{section:evolutionary}, we computed the evolutionary Khovanov homology of the corresponding knot diagram of the SARS-CoV-2 frameshifting pseudoknot. We obtained the corresponding barcode information, as shown in Figure \ref{figure:knots_barcode}. Note that the knot in Figure \ref{figure:knots_protein}\textbf{b} is unknotted, and its Khovanov homology is trivial. However, Figure \ref{figure:knots_barcode} shows that its evolutionary Khovanov homology is non-trivial, with four bars. Here, since the dimensions of the generators remain unchanged during the evolution, but their degrees change, we use the vertical axis to represent the degree. We use polyline segments to indicate the changes in the degrees of these generators.
\begin{figure}[htb!]
	\centering
	\includegraphics[width=0.75\textwidth]{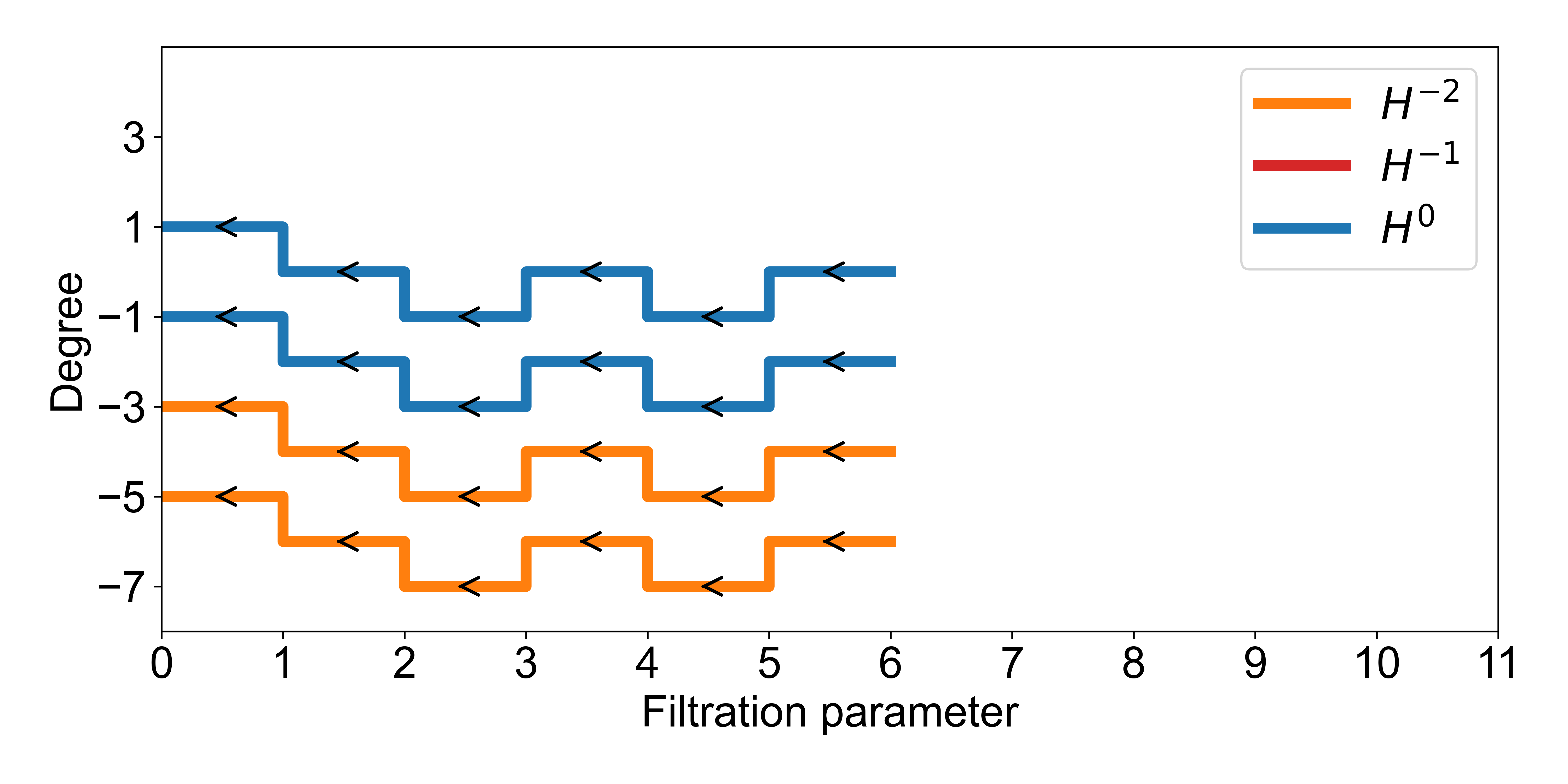}\\
	\caption{The barcode of the evolutionary Khovanov homology of the corresponding knot diagram of the SARS-CoV-2 frameshifting pseudoknot.}\label{figure:knots_barcode}
\end{figure}

This work serves as a crucial tool for understanding the intricate entanglements of the pseudoknot, offering insights into how topological variations might influence RNA functionality. This methodology not only underscores the significance of structural topology in RNA biology but also provides a framework for exploring the dynamic properties of other complex RNA molecules.

\section{Conclusion}


In this paper, we introduce evolutionary Khovanov homology in Euclidean space to study the topological invariants of real-world links at various geometric scales.
By performing systematical smoothing transformations of crossings, we transform the original link into a family of links, thereby obtaining richer geometry-informed topological  information at various filtration scales. As a result, even unknotted links may have non-trivial evolutionary Khovanov homology. Therefore, evolutionary Khovanov homology not only  characterizes the global topological structure of knots or links under knot equivalence but also captures their geometric shape. Furthermore, we  employ barcodes or persistence diagrams to depict the topological features of evolutionary Khovanov homology. It is noteworthy that the generators of evolutionary Khovanov homology not only provide persistent information but also possess degree information, which holds potential value in aiding the study of knots and links in science and engineering. Furthermore, we apply evolutionary Khovanov homology to investigate the topology and geometry of a SARS-CoV-2 frameshifting pseudoknot. The corresponding knot diagram of the abstract knot of the  pseudoknot is unknotted, resulting in trivial Khovanov homology. However, the barcode of the evolutionary Khovanov homology shows four distinct bars, which include information about changes in degree. This demonstrates that evolutionary Khovanov homology can provide a novel characterization of curved data in practical applications.

Evolutionary Khovanov homology introduces   multiscale analysis into Khovanov homology by considering the metric. It opens a new direction in geometric topology and will stimulates further developments in the field.
Additionally, this work represents an early attempt to apply advanced knot theory and geometric topology  to the quantitative analysis of curved data.
For example, the proposed approach can be extended to persistent Laplacian \cite{chen2021evolutionary,wang2020persistent,wei2023persistent} and  interaction \cite{liu2024persistent} types of formulations.  We hope that this study will open a new area in data science and machine learning. Finally, we envision both knot feature-based deep neural networks and knot theory-enabled large language models, facilitated by computational algorithm developments.

\section{Acknowledgments}
This work was supported in part by NIH grants R01GM126189, R01AI164266, and R35GM148196, National Science Foundation grants DMS2052983, DMS-1761320, and IIS-1900473, NASA  grant 80NSSC21M0023,   Michigan State University Research Foundation, and  Bristol-Myers Squibb  65109.

\bibliographystyle{plain}  
\bibliography{Reference}

\end{document}